\documentclass[11pt]{amsart}

\usepackage[T1]{fontenc}
\usepackage[utf8]{inputenc}
\usepackage{amsmath,amssymb,amsthm}
\usepackage{booktabs}
\usepackage[expansion=false]{microtype}
\usepackage[colorlinks=true,linkcolor=blue,citecolor=blue,urlcolor=blue]{hyperref}

\usepackage{xcolor}
\newif\ifreview
\reviewfalse
\DeclareRobustCommand{\review}[1]{\ifreview{\color{red}#1}\else#1\fi}
\allowdisplaybreaks

\newcommand{\F}{\mathbb{F}}
\newcommand{\PP}{\mathbb{P}}
\newcommand{\A}{\mathbb{A}}
\newcommand{\OO}{\mathcal{O}}
\newcommand{\mm}{\mathfrak{m}}
\newcommand{\tX}{\widetilde{X}}
\newcommand{\tC}{\widetilde{C}}
\newcommand{\tO}{\widetilde{\OO}}
\DeclareMathOperator{\Sing}{Sing}
\DeclareMathOperator{\gon}{gon}
\DeclareMathOperator{\divi}{div}
\DeclareMathOperator{\Con}{Con}

\theoremstyle{plain}
\newtheorem{theorem}{Theorem}[section]
\newtheorem{proposition}[theorem]{Proposition}
\newtheorem{lemma}[theorem]{Lemma}
\newtheorem{corollary}[theorem]{Corollary}

\theoremstyle{definition}
\newtheorem{definition}[theorem]{Definition}
\newtheorem{example}[theorem]{Example}

\theoremstyle{remark}
\newtheorem{remark}[theorem]{Remark}

\begin{document}

\title[Maximal singular curves over finite fields]{Maximal singular curves over finite fields from elliptic and hyperelliptic curves}

\author[C.~Hilario]{Cesar Hilario}
\address{Departamento de Ciencias, Pontificia Universidad Cat\'olica del Per\'u, Av. Universitaria 1801, San Miguel 15088, Lima, Peru}
\email{cesar.hilario.pm@gmail.com}

\author[R.~Salom\~ao]{Rodrigo Salom\~ao}
\address{Departamento de Matem\'atica Aplicada, Universidade Federal Fluminense, Niter\'oi, RJ, Brazil}
\email{rsalomao@id.uff.br}

\author[R.~V.~Martins]{Renato Vidal Martins}
\address{Departamento de Matem\'atica, ICEx, Universidade Federal de Minas Gerais, Belo Horizonte, MG, Brazil}
\email{rvidalmartins@gmail.com}


\begin{abstract}
We construct explicit families of maximal singular curves over finite fields. A maximal \emph{singular} curve is a curve that attains the Aubry--Perret bound, which is the natural extension of the Hasse--Weil--Serre bound on the maximum number of rational points on a smooth curve over a finite field. Starting from a smooth elliptic or hyperelliptic curve \( \tilde{C} \) over \(\mathbb{F}_{q}\) (\(q\) odd), we generate singular curves \(C\) with non-split singularities or cusps.  In our approach we use linear projections of geometric embeddings, and we apply Stöhr's embedding of hyperelliptic Gorenstein curves and the Rosa–Stöhr theory of trigonal Gorenstein curves. Some of our models attain the
Aubry--Perret bound, whenever their normalization is maximal, and the construction is applied to the Tafazolian and Tafazolian–Torres smooth maximal curves to produce the desired family. Finally, we determine the gonality over $\mathbb{F}_q$ of most of the curves we construct. 
\end{abstract}

\maketitle


\section{Introduction}\label{sec:intro}

Throughout this paper, a \emph{curve} is a complete, geometrically irreducible, algebraic curve defined over the finite field $\F_q$ with $q$ elements. Singular curves over finite fields arise naturally in several problems in arithmetic geometry and its applications. They appear, for instance, in the geometric construction of error-correcting codes.

From a more intrinsic point of view, the zeta function of a singular curve was studied by Aubry and Perret \cite{AP96,AP04} and by Z\'u\~niga-Galindo \cite{Zu98}, and their results show that the arithmetic of a singular curve is governed by that of its normalization together with discrete invariants of the singular points.

Let $C$ be a curve over $\F_q$, let $\nu\colon\tC\to C$ be its normalization, and denote by $g$ the genus of $\tC$, called the \emph{geometric genus} of $C$, and by $\pi$ the \emph{arithmetic genus} of $C$. The difference $\delta:=\pi-g$ is the sum of the singularity degrees of the (finitely many) singular points on $C$. Aubry and Perret \cite{AP96} proved that the number of rational points on $C$ satisfies a Hasse--Weil-type inequality which, combined with the Serre bound on the normalization, gives
\begin{equation}\label{eq:AP}
\#C(\F_q)\ \leq\ q+1+g\,\lfloor 2\sqrt{q}\rfloor+\pi-g .
\end{equation}
Z\'u\~niga-Galindo \cite{Zu98} subsequently refined this estimate by taking into account the structure of the generalized Jacobian of $C$. Following \cite{AI15}, we call a curve attaining the bound \eqref{eq:AP} \emph{maximal}. When $C$ is smooth we have $\pi=g$, and \eqref{eq:AP} recovers the bound in the classical setting.

While smooth maximal curves have been investigated intensively, explicit families of \emph{singular} maximal curves appear to be comparatively scarce in the literature. Fukasawa, Homma and Kim \cite{FHK12} studied the rational plane curve of degree $q+1$ given by the image of $\PP^1\ni(s:t)\mapsto(s^{q+1}:s^qt+st^q:t^{q+1})\in\PP^2$, nowadays called the Ballico--Hefez curve, and showed that it attains the bound \eqref{eq:AP} with $g=0$ and $\pi=(q^2-q)/2$; Ballico \cite{Bal} extended this construction to non-rational curves by means of seminormal curves with axial singularities.  These constructions provide precedents for the
explicit projective realizations developed here.

On the other hand, Aubry and Iezzi \cite{AI15,AI17} developed intrinsic constructions with prescribed singularities and
carried out a systematic study of the maximum number $N_q(g,\pi)$ of rational points on a curve of geometric genus $g$ and arithmetic genus $\pi$, characterizing in \cite[Theorem~5.3]{AI15} the pairs $(g,\pi)$ for which the bound $N_q(g,\pi)\leq N_q(g)+\pi-g$ is attained. A structural result of theirs \cite[Proposition~5.2]{AI15} asserts that every singular point on a curve attaining this bound is
rational, has singularity degree one, and lies under a unique closed point of
degree two on the normalization. 
This characterization motivates us to introduce the notion of a \emph{non-split singularity} (see Definition~\ref{def:nonsplit}). For plane curves, the typical example is an ordinary double point whose tangent cone is irreducible over $\F_q$ and splits into two Frobenius-conjugate lines over $\F_{q^2}$, hence the terminology. 

The purpose of the present paper is to construct explicit maximal singular curves all of whose singularities are non-split, with elliptic or hyperelliptic normalization, and to study the geometry of the curves so obtained, notably by computing their gonality over the ground field.

Our main construction is carried out in Section~\ref{sec:construction}. Here we build maximal singular curves out of three types of maximal \review{smooth curves, namely elliptic}, hyperelliptic with $2\leq g \leq 3$, and hyperelliptic with $g\geq 4$. To give an idea of how this works, let $q$ be odd and let $\tC\colon y^2=f(x)$ be a smooth curve of genus $g\geq 1$ over $\F_q$, with $f\in\F_q[x]$ squarefree of degree $2g+1$, and let $P$ be a closed point on $\tC$ of degree two, which lies under a point $(\alpha,\beta)\in\tC\otimes\F_{q^2}$ with $f(\alpha)\neq 0$ and $\alpha\notin\F_q$. 

In Theorem~\ref{thm:construction} we produce an explicit morphism $\Phi\colon\tC\to\PP^3$, obtained by
an explicit linear projection of the embedding used by Stöhr \cite{St99} when $g\geq 2$, and of another embedding \review{when $g=1$}, whose image $C=\Phi(\tC)$ is a curve over $\F_q$ with normalization $\tC$ and with non-split singularities. 
When $\beta\notin\F_q$ there is a unique non-split singularity (lying under $P$), and when $\beta\in\F_q$ there are exactly two non-split singularities (lying under $P$ and the hyperelliptic conjugate $P^-$). For $2\leq g\leq 3$ these are the only singularities on $C$, hence $\pi=g+k$ and $\#C(\F_q)=\#\tC(\F_q)+k$ with $k\in\{1,2\}$ the number of nodes, and so $C$ attains the Aubry--Perret bound \eqref{eq:AP} over $\F_q$ whenever $\tC$ is maximal over $\F_q$.

For $g=1$  and $g\geq 4$, on the other hand, the projected curve acquires one further singularity $Q_\infty$ at the image of the point at infinity of $\tC$; this singularity is a rational unibranch point, actually an ordinary cusp when $g=1$. 
To achieve maximality for these genera \review{two further constructions are needed. For $g=1$}, an additional modified projection yields a curve $C\subset\PP^3$ whose unique singularity is a non-split singularity (see Theorem~\ref{thm:elliptic}); and for $g\geq 4$ the Rosa--St\"ohr theory of trigonal Gorenstein curves \cite{RS02} produces, out of the same data, a canonical curve $C$ on a cone in $\PP^{g}$ whose unique singularity is a non-split singularity at the vertex (see Theorem~\ref{thm:cone}). 

For the nodal models just described, the curve $C$ attains the Aubry--Perret bound over $\F_q$ whenever $\tC$ is maximal over $\F_q$, and we apply the constructions to the maximal hyperelliptic curves $y^2=x^m+1$ of Tafazolian \cite{Taf12} and to the Artin--Schreier family $y^2=x^q+x$ of Tafazolian--Torres \cite{TT17}, besides working out explicit examples over $\F_{25}$, $\F_{81}$ and $\F_3$.

Section~\ref{sec:gonality} studies the gonality
over $\F_q$ of the curves constructed above.
For a smooth projective curve, gonality over
the ground field is the minimum degree of a
nonconstant morphism to $\PP^1$, equivalently
the minimum degree of a pencil defined over
that field. Over finite fields, this invariant
also has applications to AG codes \cite{P}.

For singular curves, a broader theory of linear
systems allows torsion-free sheaves of rank one
in place of invertible sheaves. This framework
is developed in the work of Altman--Kleiman
\cite{AK} and Hartshorne \cite{H86}. We use the
language of divisors represented by products
of local fractional ideals, following
St\"ohr \cite{St93} and Carvalho \cite{Car99}.
In this setting, pencils may have non-removable
base points, so their minimum degree need not
coincide with the minimum degree of a morphism
to $\PP^1$. Rosa--St\"ohr \cite{RS02} exhibit
this phenomenon for trigonal Gorenstein curves
\review{on cones, where the rulings} cut out a degree-three
pencil with a non-removable base point at the
vertex. This geometry occurs in the cone
models constructed here.

The table below summarizes the gonality
computations of Theorems~\ref{thm:gonality}
and~\ref{thm:char}.

\begin{table}[htbp]
\centering
\caption{Gonality of the curves constructed in
Section~\ref{sec:construction}.}
\label{tab:gonality}
\small
\begin{tabular}{|c|c|c|c|c|c|}
\hline
curve of thm. & case & sing. & $\pi$ & gon.\ & examples\\
\hline
\ref{thm:elliptic} & $g=1$, any $\beta$ & $1$ node & $2$ & $2$ & \ref{ex:g1max}\\
\hline
 & $g=1$, $\beta\notin\F_q$ & node, cusp & $3$ & $3$ & \ref{ex:g1cusp}\\
 & $g=1$, $\beta\in\F_q$ & $2$ nodes, cusp & $4$ & $3$ & ---\\
\ref{thm:construction} & $g\in\{2,3\}$, $\beta\notin\F_q$ & $1$ node & $g+1$ & $3$ & \ref{ex:g2max}\\
 & $g=2$, $\beta\in\F_q$ & $2$ nodes & $4$ & $3$, $4$ & \ref{ex:F3trigonal}, \ref{ex:F3tetragonal}\\
 & $g=3$, $\beta\in\F_q$ & $2$ nodes & $5$ & $4$ & ---\\
\hline
\ref{thm:cone} & $g\geq4$, any $\beta$ & $1$ node & $g+1$ & $3$ & ---\\
\hline
\end{tabular}
\end{table}

In all aforementioned cases, the central idea is to explore (via Proposition~\ref{prop:degformula}) how the local structure of the singularities, combined with the properties of the normalization, determine the degree of the linear systems that compute gonality. More precisely, we conduct a fine analysis of the local rings of the singularities together with a study of how the group structure of elliptic curves or the divisor class group affects the existence of rational functions with prescribed pole divisors. Both approaches are crucial to determine the degree of the divisors on the singular curve.

\subsection*{Acknowledgements.} The third named author is partially supported by CNPq grant number 308950/2023-2 and FAPESP grant number 2024/15918-8.

\section{Definitions and basic results}\label{sec:basics}

\subsection{Singular curves, rational points and the Aubry--Perret bound}\label{subsec:setup}

As in the introduction, for simplicity, by a \emph{curve} we mean a complete, geometrically irreducible, algebraic curve defined over $\F_q$. Our notation follows \cite{Zu98} and \cite{AI15}. Let $X$ be a curve over $\F_q$ with function field $\F_q(X)$, and let $\nu\colon\tX\to X$ be its normalization. Both $\tX$ and $\nu$ are defined over $\F_q$, and $\F_q(\tX)=\F_q(X)$. For a closed point $Q\in X$ we denote by $\OO_{X,Q}$ the local ring of $X$ at $Q$, by $\mm_{X,Q}$ its maximal ideal, by $\F_q(Q):=\OO_{X,Q}/\mm_{X,Q}$ its residue field and by $\deg Q:=[\F_q(Q):\F_q]$ its degree. A point of degree one is called \emph{rational}. The integral closure of $\OO_{X,Q}$ in $\F_q(X)$ is
\[
\tO_Q=\bigcap_{P\mapsto Q}\OO_{\tX,P},
\]
where the intersection runs over the (finitely many) points $P\in\tX$ lying over $Q$, which are called the \emph{branches} of $X$ at $Q$. The point $Q$ is said to be \emph{unibranch over $\mathbb{F}_q$} if there is exactly one such branch $P$. We use "unibranch" in this sense throughout; it does not mean geometrically unibranch. In particular, a non-split singularity (as in Subsection \ref{subsec:nonsplit}) has one closed point
of degree two above it and two geometric branches. The \emph{singularity degree} or \emph{$\delta$-invariant} of $Q$ is
$
\delta_Q:=\dim_{\F_q}\tO_Q/\OO_{X,Q}$,
so that $\delta_Q=0$ if and only if $Q$ is a nonsingular point. Setting $\delta:=\sum_{Q\in\Sing X}\delta_Q$, the \emph{arithmetic genus} of $X$ is equal to $\pi=g+\delta$, where $g$ is the genus of $\tX$ (see \cite[Section~2]{AI15}).

The following elementary relation between the rational points on $X$ and those on $\tX$ will be used repeatedly; see the proof of \cite[Proposition~2.3]{AP96}. For a rational singular point $Q\in X$, letting $a_Q:=\#\{P\in\tX(\F_q):\nu(P)=Q\}$ denote the number of rational branches of $X$ at $Q$, the following holds
\begin{equation}
\label{lem:pointcount}
\#X(\F_q)=\#\tX(\F_q)+\sum_{Q\in\Sing X\cap X(\F_q)}\bigl(1-a_Q\bigr).
\end{equation}


Our starting point is a result proved by Aubry and Perret in \cite[Proposition~2.3]{AP96}, which combined with the Serre bound $\#\tX(\F_q)\leq q+1+g\lfloor 2\sqrt{q}\rfloor$ (see \cite[Theorem~5.3.1]{Sti09}) yields the following bound
\begin{equation}
\label{thm:AP}
\#X(\F_q)\ \leq\ q+1+g\,\lfloor 2\sqrt{q}\rfloor+\pi-g,
\end{equation}
where $X$ is a curve over $\F_q$ of geometric genus $g$ and arithmetic genus $\pi$.
\begin{definition}\label{def:maximal}
Following \cite[Definition 5.1.(iii)]{AI15}, we call a curve $X$ over $\F_q$ attaining the bound~\eqref{thm:AP} a \emph{maximal curve} over $\F_q$.
\end{definition}

When $X$ is singular we have $\pi>g$, and maximality forces the normalization $\tX$ to have many rational points and the singularities of $X$ to convert closed points of higher degree on $\tX$ into additional rational points on $X$, in the sense of~\eqref{lem:pointcount}. In the next subsection we will make this mechanism precise.

\subsection{Non-split singularities}\label{subsec:nonsplit}

Our constructions will be guided by the following structural result of Aubry and Iezzi, stated here in a form adapted to our purposes.

\begin{proposition}[{\cite[Proposition~5.2]{AI15}}]
\label{prop:AIcharacterization}
If $X$ is maximal, then so is $\widetilde{X}$. Moreover, every singular point $Q$ of $X$ satisfies the following properties: (a) $Q$ is rational; (b) $\delta_Q=1$; (c) $Q$ is unibranch over $\mathbb{F}_q$ (that is, $\# \nu^{-1}(Q)=1$) and its branch $P$ in $\widetilde X$ has degree 2.
\end{proposition}

This characterization motivates the following definition.

\begin{definition}\label{def:nonsplit}
A point $Q$ on a curve $X$ over $\F_q$ is called a \emph{non-split singularity} if it satisfies (a), (b), and (c) above. 
\end{definition}

\begin{remark}\label{rem:nonsplitplane}
For plane curves, the typical example of a non-split singularity is an ordinary double point $Q$ whose tangent cone $ax^2+bxy+cy^2$ is irreducible in $\F_q[x,y]$, and hence splits into two distinct conjugate linear factors in $\F_{q^2}[x,y]$. Indeed, over $\overline{\F}_q$ the point $Q$ is an ordinary node, so $\delta_Q=1$ and there are two geometric branches; since these branches are exchanged by Frobenius (as their tangent directions are) they are glued into a single closed point of degree two on the normalization, so that $X$ is unibranch at $Q$ over $\F_q$. 
Since a non-split singularity has $a_Q=0$, by~\eqref{lem:pointcount} it increases the number of rational points by one with respect to the normalization, hence if all singularities of $X$ are non-split one has
$\#X(\F_q)=\#\tX(\F_q)+\#\Sing X $.
\end{remark}

\begin{example}\label{ex:cubic}
Let $q$ be odd and consider the plane cubic $C_q$ over $\F_q$, defined by the polynomial $(x^2+y^2)z+xy^2$. 
The cubic $C_q$ is irreducible 
and has arithmetic genus $\pi=1$. 
The Jacobian criterion gives $\Sing C_q=\{Q_0\}$, where $Q_0=(0:0:1)$ is an ordinary node, whence $\delta_{Q_0}=1$ and $g=0$. The tangent cone at $Q_0$ is defined by the polynomial $x^2+y^2$, which is irreducible over $\F_q$ exactly when $-1$ is a non-square in $\F_q$, that is, exactly when $q\equiv 3\pmod 4$. So for $q\equiv 3\pmod 4$ the node is non-split by Remark~\ref{rem:nonsplitplane}, 
and therefore $\#C_q(\F_q)=(q+1)+1$, i.e., $C_q$ is maximal.

\end{example}

We now establish the local-ring characterization of non-split singularities, which will be the key algebraic tool both in the constructions in Section~\ref{sec:construction} and the gonality computations in Section~\ref{sec:gonality}.

\begin{lemma}\label{lem:localchar}
Let $Q$ be a non-split singularity on a curve $X$ over $\F_q$ and let $P\in\tX$ be the point lying over $Q$. Then, inside $\OO_{\tX,P}$ we have $$\OO_{X,Q}=\F_q+\mm_{\tX,P}.$$
\end{lemma}

\begin{proof}
Note that the rationality of $Q$, i.e., $\F_q(Q)=\F_q$, is equivalent to the equality
$\OO_{X,Q}=\F_q+\mm_{X,Q}$,
which readily implies $\OO_{X,Q}\subseteq\F_q+\mm_{\tX,P}$.
Likewise, in the chain of $\F_q$-vector spaces 
$\OO_{X,Q}\subseteq\F_q+\mm_{\tX,P}\subseteq \OO_{\tX,P}$ the latter inclusion is proper, because $P$ has degree $2$ and therefore is not rational.
Thus the proof will be complete once we check that the $\F_q$-vector space $\OO_{\tX,P}/\OO_{X,Q}$ has dimension $1$.
This holds true because the singularity degree $\delta_Q=\dim_{\F_q}\tO_Q\big/\OO_{X,Q}$ is equal to $1$ and $X$ is unibranch at $Q$, so that $\tO_Q=\OO_{\tX,P}$.
\end{proof}



\subsection{Rational unibranch singular points and value semigroups}\label{subsec:semigroups}

Besides non-split singularities, some of the curves constructed in Section~\ref{sec:construction} exhibit singular points \review{of a different nature, namely rational unibranch points} whose branch is again rational. Their basic invariant is a numerical semigroup.

\begin{definition}\label{def:semigroup}
Let $Q\in X$ be a unibranch singular point with branch $P\in\tX$, and suppose $\deg Q=\deg P=1$. Then $\tO_Q=\OO_{\tX,P}$ is a discrete valuation ring with valuation $v:=v_P$ and residue field $\F_q$. The \emph{value semigroup} of $Q$ is
\[
S(Q):=v\bigl(\OO_{X,Q}\setminus\{0\}\bigr)\subseteq\mathbb{N} .
\]
It contains $0$, it is closed under addition, and its complement in $\mathbb{N}$ is finite, because $\OO_{X,Q}$ contains the conductor $(\OO_{X,Q}:\tO_Q)$, which is a nonzero ideal of $\tO_Q$. The point $Q$ is called an \emph{ordinary cusp} when $S(Q)=\langle 2,3\rangle$.
\end{definition}

\begin{lemma}\label{lem:gaps}
With the above notation, let $D_Q\subseteq\F_q(X)$ be a fractional ideal of $\OO_{X,Q}$ with $D_Q\supseteq\OO_{X,Q}$, and set $S(D_Q):=v\bigl(D_Q\setminus\{0\}\bigr)$. Then
\[
\dim_{\F_q}D_Q\big/\OO_{X,Q}=\#\bigl(S(D_Q)\setminus S(Q)\bigr).
\]
In particular $\delta_Q=\#\bigl(\mathbb{N}\setminus S(Q)\bigr)$ is the number of gaps of $S(Q)$ and, for a rational function $f\in\F_q(X)^{\times}$ with $n:=\max\bigl(0,-v(f)\bigr)$ one has
\[
n-\delta_Q\;\leq\;\dim_{\F_q}\bigl(\OO_{X,Q}+f\OO_{X,Q}\bigr)\big/\OO_{X,Q}\;\leq\;n+\delta_Q .
\]

\end{lemma}

\begin{proof}
Write $\OO:=\OO_{X,Q}$, and for each $m\in\mathbb{Z}$ let $\tO_{\geq m}:=\{h\in\F_q(X)\,:\,v(h)\geq m\}$. Each quotient $\tO_{\geq m}/\tO_{\geq m+1}$ is one-dimensional over the residue field $\F_q$. Choose $m_0$ with $D_Q\subseteq\tO_{\geq m_0}$ and $c\geq 0$ with $\tO_{\geq c}\subseteq\OO$, and consider the filtration
\[
D_Q=F_{m_0}\supseteq F_{m_0+1}\supseteq\dots\supseteq F_c=\OO,
\] 
$F_m:=\OO+\bigl(D_Q\cap\tO_{\geq m}\bigr)$.

Fix $m$. If $m\in S(D_Q)\setminus S(Q)$, pick $d\in D_Q$ with $v(d)=m$; \review{its class generates $F_m/F_{m+1}$. Indeed,} any $d'\in D_Q$ with $v(d')=m$ is congruent to a scalar multiple of $d$ modulo $D_Q\cap\tO_{\geq m+1}$, since the residue field is $\F_q$, and the class of $d$ is nonzero, for an identity $d=o+d'$ with $o\in\OO$ and $d'\in D_Q\cap\tO_{\geq m+1}$ would force $v(o)=m\in S(Q)$. If instead $m\in S(Q)$, then every $d\in D_Q$ with $v(d)=m$ is congruent modulo $\tO_{\geq m+1}$ to an element of $\OO$ with valuation $m$, so $F_m=F_{m+1}$; and the same conclusion holds trivially when $m\notin S(D_Q)$. Therefore $\dim_{\F_q}F_m/F_{m+1}$ equals $1$ if $m\in S(D_Q)\setminus S(Q)$ and $0$ otherwise, and the first assertion follows by summing over $m$. Taking $D_Q=\tO_Q$, whose value set is $\mathbb{N}$, gives $\delta_Q=\#(\mathbb{N}\setminus S(Q))$. \review{Finally, take $D_Q=\OO+f\OO$. On the one hand} $D_Q\subseteq\tO_{\geq -n}$, so $S(D_Q)\setminus S(Q)\subseteq\mathbb{Z}_{\geq-n}\setminus S(Q)$, a set of cardinality $n+\delta_Q$; on the other hand $S(D_Q)\supseteq v(f)+S(Q)$, and already the negative part $\bigl(v(f)+S(Q)\bigr)\cap\mathbb{Z}_{<0}$, which has cardinality $\#\bigl(S(Q)\cap[0,n)\bigr)\geq n-\delta_Q$, is disjoint from $S(Q)$.
\end{proof}

\subsection{Points of degree two on elliptic and hyperelliptic curves}\label{subsec:degtwo}
We now fix the notation for degree-two points on elliptic and hyperelliptic curves,
and assume henceforth that $q$ is odd.

Let $\tC$ be a geometrically irreducible and smooth curve over $\F_q$ of genus $g\geq 1$ which is elliptic or hyperelliptic. We assume throughout that $\tC$ admits a plane affine model
\begin{equation}\label{eq:model}
\tC\colon\ y^2=f(x),\qquad f\in\F_q[x]\ \text{squarefree},\quad \deg f=2g+1 ,
\end{equation}
so that the closed points of $\tC$ correspond to the places of the function field $K:=\F_q(x,y)$ over $\F_q$, which is a separable extension of degree two of the rational function field $\F_q(x)$. This assumption can be phrased purely in terms of the function field $K$. The ramified places of $K|\F_q(x)$ are the zeros of $f$, together with the infinite place of $\F_q(x)$ exactly when $\deg f$ is odd; thus a model \eqref{eq:model} with $\deg f$ odd exists if and only if some place of degree one of $\F_q(x)$ ramifies in $K$, since such a place can be taken to the infinite place by a fractional linear change of the generator $x$, which does not alter the subfield $\F_q(x)$. For $g\geq 2$ the subfield $\F_q(x)$ \review{is moreover intrinsic to $K$, being its unique} rational subfield of index two (see \cite[Chapter~VI.2]{Sti09}). For $g=1$, the existence of a model \eqref{eq:model} amounts to the existence of a place of degree one of $K$. All the examples treated in this paper are of this form.

Since $\deg f$ is odd, the infinite place of $\F_q(x)$ is totally ramified in $K$, so there is a unique place $P_\infty$ of $K$ lying over it; it has degree one, and
\begin{equation}\label{eq:valinfty}
v_{P_\infty}(x)=-2,\qquad v_{P_\infty}(y)=-(2g+1).
\end{equation}
Setting $\xi:=1/x$ and $\gamma:=y\,x^{-(g+1)}$ we get $v_{P_\infty}(\xi)=2$ and
$v_{P_\infty}(\gamma)=v_{P_\infty}(y)+2(g+1)=1$,
so that $\gamma$ is a local parameter at $P_\infty$. These functions will be used in Section~\ref{sec:construction} to analyse the behaviour of our construction at infinity.

We now fix the notation we will employ throughout Sections~\ref{sec:construction} and~\ref{sec:gonality}. Let $P\in\tC$ be a closed point of degree two whose restriction to the subfield $\F_q(x)$ is a place of degree two; equivalently, the residue $\alpha:=x(P)$ of the function $x$ in the residue field of $P$ satisfies
\[
\alpha\in\F_{q^2}\setminus\F_q .
\]
(Closed points of degree two with \emph{rational} $x$-coordinate also exist --- they have $\beta\in\F_{q^2}\setminus\F_q$ with $\beta^2=f(\alpha)$ a non-square in $\F_q$ --- but they will play no role in this paper.) 

By \cite[Lemma~5.1.9]{Sti09}, applied with $r=m=2$, the closed point $P$ splits on the base change $\tC\otimes_{\F_q}\F_{q^2}$ into two distinct rational points that are exchanged by the Frobenius automorphism, namely 
$P_1=(\alpha,\beta)$ and $P_2=(\alpha^q,\beta^q)$, 
where $\beta\in\F_{q^2}$ satisfies $\beta^2=f(\alpha)$. Both possibilities $\beta\in\F_q$ and $\beta\in\F_{q^2}\setminus\F_q$ 
occur, and they will lead to genuinely different geometries in Sections~\ref{sec:construction} and~\ref{sec:gonality}. Note that $\beta\in\F_q$ if and only if $f(\alpha)$ belongs to $\F_q$ and is a square in $\F_q^{\times}$ (or $f(\alpha)=0$).

We denote by $P^-\in\tC$ the image of $P$ under the nontrivial automorphism of $K$ over $\F_q(x)$, which maps $y$ to $-y$; it is again a closed point of degree two with $x(P^-)=\alpha$, and the two rational points on $\tC\otimes\F_{q^2}$ lying over $P^-$ are
$P_1^-=(\alpha,-\beta)$ and $P_2^-=(\alpha^q,-\beta^q)$.
If $f(\alpha)\neq 0$, then $\beta\neq 0$ and hence $P\neq P^-$.


\section{Maximal curves with elliptic or hyperelliptic normalization}\label{sec:construction}

We keep the data $(\tC,P)$ of Subsection~\ref{subsec:degtwo}, including $\alpha\notin\F_q$ and $f(\alpha)\neq0$, and write $d=\deg f=2g+1$.
Starting from the pair $(\tC,P)$, we construct explicit curves over $\F_q$ with normalization $\tC$ and with a non-split singularity lying under $P$ or under $P^-$. All of these constructions use the data \review{introduced in Subsection~\ref{subsec:interpolation}}.

\subsection{The interpolation data}\label{subsec:interpolation}

The minimal polynomial of $\alpha$ over $\F_q$ is given by
\begin{equation}\label{eq:phi}
\varphi(x):=(x-\alpha)(x-\alpha^q)=x^2-\sigma x+\eta,
\qquad
\sigma:=\alpha+\alpha^q,\quad \eta:=\alpha^{1+q}.
\end{equation}
In particular its discriminant $\sigma^2-4\eta$ is a non-square in $\F_q^{\times}$, and hence it is nonzero, since $q$ is odd.

\begin{lemma}\label{lem:lambda}
There is a unique polynomial $\lambda(x)=\ell_1x+\ell_0$ of degree at most one with $\lambda(\alpha)=\beta$
and
$\lambda(\alpha^q)=\beta^q$,
whose coefficients $\ell_1,\ell_0$ lie in $\F_q$. Moreover, $\beta\in\F_q$ if and only if $\lambda$ is the constant polynomial $\beta$.
\end{lemma}

\begin{proof}
Since $\alpha\neq\alpha^q$, Lagrange interpolation over $\F_{q^2}$ gives existence and uniqueness, with 
$\ell_1=(\beta^q-\beta)/(\alpha^q-\alpha)$, 
$\ell_0=\beta-\ell_1\alpha$.

As $\alpha^{q^2}=\alpha$ and $\beta^{q^2}=\beta$, raising to the $q$-th power permutes the interpolation conditions, so $\ell_1^q=\ell_1$ and $\ell_0^q=\ell_0$, that is, $\ell_1,\ell_0\in\F_q$. Finally, $\ell_1=0$ if and only if $\beta^q=\beta$, in which case $\lambda=\ell_0=\beta$.
\end{proof}

\begin{lemma}\label{lem:h}
The polynomial $\varphi$ divides $f-\lambda^2$ in $\F_q[x]$. The quotient
\begin{equation}\label{eq:h}
h(x):=\frac{f(x)-\lambda(x)^2}{\varphi(x)}\in\F_q[x]
\end{equation}
has degree $d-2=2g-1$ and the same leading coefficient as $f$.
\end{lemma}

\begin{proof}
The first assertion follows from $(f-\lambda^2)(\alpha)=f(\alpha)-\beta^2=0$. 
As $\deg\lambda^2\leq 2<d$, the leading term of $f-\lambda^2$ is that of $f$, and the degree statement follows.
\end{proof}

Writing $t:=y-\lambda(x)$, the model \eqref{eq:model} translates into the \emph{functional equation}
\begin{equation}\label{eq:functional}
t^2+2\lambda(x)\,t=\varphi(x)\,h(x) ,
\end{equation}
since $t^2+2\lambda t=(y-\lambda)^2+2\lambda(y-\lambda)=y^2-\lambda^2=f-\lambda^2$. Equation \eqref{eq:functional} will be the source of the affine equations of our curves.

\subsection{The projective model}\label{subsec:model}

We first specify the embedding from which the projection is taken. For $g\geq2$, use
\[
\iota\colon\tC\longrightarrow\PP^{g+2},\qquad
(x,y)\longmapsto(1:x:\cdots:x^{g+1}:y),
\]
as in St\"ohr's study of hyperelliptic Gorenstein curves \cite{St99}. The functions $1,x,\ldots,x^{g+1},y$ \review{form a basis of $L((2g+2)P_\infty)$, since they have} distinct pole orders and their number is $g+3$, as given by Riemann--Roch. Since $2g+2\geq2g+1$, this complete linear system is very ample. Write $(X_0:\cdots:X_{g+1}:Y)$ for its projective coordinates.

For $g=1$, we instead take the very ample complete linear system $|6P_\infty|$, with basis $1,x,x^2,x^3,y,xy$, and put
\[
\iota\colon\tC\longrightarrow\PP^5,\qquad
(x,y)\longmapsto(1:x:x^2:x^3:y:xy).
\]
Its coordinates are denoted by $(X_0:X_1:X_2:X_3:Y:Z)$. This choice includes the coordinate $x^3$ needed by the projection below.

In either case, let $\rho$ be the linear projection to $\PP^3$ defined by
\begin{equation}\label{eq:matrixM}
\begin{aligned}
U&=\eta X_0-\sigma X_1+X_2,&
W&=\eta X_1-\sigma X_2+X_3,\\
T&=-\ell_0X_0-\ell_1X_1+Y,& S&=X_0.
\end{aligned}
\end{equation}
Thus the coefficient of $Z$ in the elliptic case is zero. The composition $\rho\circ\iota$ is a priori a rational map. We write $\Phi\colon\tC\to\PP^3$ for its extension to a morphism, established in Theorem~\ref{thm:construction}(a). Its image is the curve to be studied; the projection permits explicit control of both the identified points and the possible additional singularity at infinity.

On the affine chart $\{S\neq 0\}\cong\A^3$ of $\PP^3$ we use the coordinates
$u:=U/S$, $w:=W/S$, $t:=T/S$.

\begin{theorem}\label{thm:construction}
With the notation above, the following hold.
\begin{enumerate}
\item[(a)] $\rho\circ\iota$ extends to a morphism $\Phi\colon\tC\to\PP^3$, defined over $\F_q$ and given on the chart $(x,y)$ of $\tC$ by
$\Phi(x,y)=\bigl(\varphi(x):x\varphi(x):y-\lambda(x):1\bigr)$, 
with
\[
\Phi(P_\infty)=
\begin{cases}
(0:1:0:0) & \text{if } g\leq 2,\\[2pt]
(0:0:1:0) & \text{if } g\geq 3.
\end{cases}
\]
\item[(b)] $C:=\Phi(\tC)$ is a complete, geometrically irreducible curve over $\F_q$, and $\Phi\colon\tC\to C$ is its normalization.
\item[(c)] If $\beta\notin\F_q$, then $\Sing C=\{Q\}\cup\Sigma_\infty$, where $Q=(0:0:0:1)$ is a non-split singularity whose branch is $P$ and $\Sigma_\infty$ is the set described in (e). Consequently the arithmetic genus of $C$ is $g+1+\delta_\infty$, where $\delta_\infty:=\sum_{Q'\in\Sigma_\infty}\delta_{Q'}$, and
$\#C(\F_q)=\#\tC(\F_q)+1$.

\item[(d)] If $\beta\in\F_q$, then $\Sing C=\{Q,Q^-\}\cup\Sigma_\infty$, where $Q=(0:0:0:1)$ and $Q^-=(0:0:-2\beta:1)$ are non-split singularities whose branches are $P$ and $P^-$, respectively, and $\Sigma_\infty$ is as in (e). Consequently the arithmetic genus of $C$ is $g+2+\delta_\infty$ and
$\#C(\F_q)=\#\tC(\F_q)+2$.

\item[(e)] $\Sigma_\infty=\emptyset$ if $2\leq g\leq 3$. If $g=1$ or $g\geq 4$, then $\Sigma_\infty=\{Q_\infty\}$, where $Q_\infty:=\Phi(P_\infty)$ is a rational unibranch singular point of $C$ \review{whose branch is $P_\infty$. For $g=1$} its value semigroup is $S(Q_\infty)=\langle 2,3\rangle$ --- an ordinary cusp, with $\delta_{Q_\infty}=1$ --- while for $g\geq 4$ its value semigroup contains $2g-5$, $2g-3$ and $2g+1$, and
\[
2g-6\;\leq\;\delta_{Q_\infty}\;\leq\;\#\bigl(\mathbb{N}\setminus\langle 2g-5,\,2g-3,\,2g+1\rangle\bigr).
\]
\item[(f)] If $2\leq g\leq 3$ and $\tC$ is maximal over $\F_q$, then $C$ attains the Aubry--Perret bound \eqref{eq:AP} over $\F_q$. For $g=1$ and for $g\geq 4$ \review{the bound is never attained. When} $\tC$ is maximal over $\F_q$, 
$\#C(\F_q)=q+1+g\lfloor 2\sqrt{q}\rfloor+(\pi-g)-\delta_{Q_\infty}$.

\end{enumerate}
\end{theorem}

\begin{proof}
We write $\pi$ for the arithmetic genus of $C$ and proceed in eight steps.

\emph{\review{Step 1.} the affine expression and two relations.}
Composing $\iota$ with the linear forms defining $\rho$ and dividing by $X_0$, we obtain, on the affine chart of $\tC$,
\[
u=x^2-\sigma x+\eta=\varphi(x),
\qquad
w=x^3-\sigma x^2+\eta x=x\varphi(x),
\qquad
t=y-\lambda(x) ,
\]
which proves the displayed formula in (a); since the last coordinate equals $1$, the restriction of $\Phi$ to the affine chart of $\tC$ is a morphism with image in $\{S\neq 0\}$. The three coordinate functions satisfy two relations. First, from $w=xu$ and $u=\varphi(x)$,
$w^2-\sigma uw+\eta u^2=u^2\bigl(x^2-\sigma x+\eta\bigr)=u^2\varphi(x)=u^3$.

Second, substituting $x=w/u$ in the functional equation \eqref{eq:functional} and multiplying by $u^{d-3}$ (respectively by $u$ when $d=3$) and using $\lambda(w/u)u=\ell_1w+\ell_0u$ together with Lemma~\ref{lem:h}, we obtain the two relations
\begin{equation}\label{eq:I}
w^2-\sigma uw+\eta u^2=u^3
\tag{I}
\end{equation}
and, writing $H(W,U):=U^{\,d-2}h(W/U)\in\F_q[W,U]$ for the homogenization of $h$ in degree $d-2$,
\begin{equation}\label{eq:II}
u^{\,d-3}t^2+2\ell_0u^{\,d-3}t+2\ell_1wu^{\,d-4}t=H(w,u)
\tag{II}
\end{equation}
for $d\geq 5$, which for $g=1$, i.e.\ $d=3$ and $h(x)=h_1x+h_0$, reads $ut^2+2\ell_0ut+2\ell_1wt=h_1uw+h_0u^2$. 

These identities hold on $\tC$, and hence on $C$. Throughout, $C$ is defined as the reduced image
$\Phi(\tC)$. Relations \eqref{eq:I} and
\eqref{eq:II} hold on its affine chart and are
used below to recover the inverse map on a
suitable open subset; they are not asserted
to generate its defining ideal. Indeed, their common zero locus contains the entire line $u=w=0$, whereas $C$ meets that line in only finitely many points. 

\emph{Step 2. Extension to $P_\infty$.}
By \eqref{eq:valinfty}, the coordinate functions
$u=\varphi(x)$, $w=x\varphi(x)$ and $t=y-\lambda(x)$
satisfy
$v_{P_\infty}(u)=-4$, 
$v_{P_\infty}(w)=-6$ and
$v_{P_\infty}(t)=-(2g+1)$.

For $g\leq2$, dividing the coordinate vector
$(u:w:t:1)$ by $w$ gives
$(u/w:1:t/w:1/w)$.
The three affine coordinates in the chart
$\{W\neq0\}$ have valuations
$2$, $5-2g$ and $6$, respectively.
They are therefore regular and vanish at $P_\infty$,
so $\Phi$ extends there with
$\Phi(P_\infty)=(0:1:0:0)$.

For $g\geq3$, divide instead by $t$.
The coordinates $u/t$, $w/t$ and $1/t$ in the
chart $\{T\neq0\}$ have valuations
$2g-3$, $2g-5$ and $2g+1$, respectively.
As they are positive $\Phi$ extends
with $\Phi(P_\infty)=(0:0:1:0)$.
Together with its affine expression, this defines
a morphism on all of $\tC$ over $\F_q$,
proving (a).

\emph{\review{Step 3.} the image curve.}
As $\tC$ is complete and geometrically irreducible and $\Phi$ is a non-constant morphism defined over $\F_q$ (it is non-constant because \review{$U/S=\varphi(x)$} is a non-constant function) its image $C=\Phi(\tC)$ is a complete, geometrically irreducible curve over $\F_q$.

\emph{Step 4. An isomorphism onto an open subset.}
Set
$U_{\tC}:=\tC\setminus\{P,P^-,P_\infty\}$
and $U_C:=C\cap\{S\neq0,\ U\neq0\}$.
Since the affine zeros of $\varphi(x)$ are
precisely $P$ and $P^-$, and $P_\infty$ is
the only point mapped into $\{S=0\}$ by Step~2,
we have $\Phi^{-1}(U_C)=U_{\tC}$.

The restriction $\Phi|_{U_{\tC}}$ has inverse
$U_C\rightarrow U_{\tC}$ given by
\[
(u,w,t)\mapsto
\left(\frac{w}{u},\,
t+\lambda\left(\frac{w}{u}\right)\right).
\]
Indeed, relations \eqref{eq:I} and \eqref{eq:II}
give $\varphi(w/u)=u$ and
$\bigl(t+\lambda(w/u)\bigr)^2=f(w/u)$,
so this is a well-defined morphism into $U_{\tC}$.
The identities $u=\varphi(x)$, $w=xu$ and
$t=y-\lambda(x)$ show that the two morphisms
are inverse to each other.
Thus $\Phi|_{U_{\tC}}$ is an isomorphism onto
$U_C$, and $\Phi$ is birational onto $C$.

\emph{\review{Step 5.} the differential of $\Phi$ on the affine part and, for $2\leq g\leq 3$, at $P_\infty$.}
We check that $d\Phi$ is nonzero at every affine point of $\tC\otimes\overline{\F}_q$ and, when $2\leq g\leq 3$, also at $P_\infty$; the behaviour at $P_\infty$ for the remaining genera is analysed in Step~8. Since $q$ is odd, $x$ is a separating variable of $K$.

At a geometric affine point $R=(x_0,y_0)$ with $y_0\neq0$, the function $x-x_0$ is a local parameter, so $dx$ generates the cotangent space and
$dU=\varphi'(x)\,dx=(2x-\sigma)\,dx,$ and
$dW=\bigl(\varphi(x)+x\varphi'(x)\bigr)\,dx=(3x^2-2\sigma x+\eta)\,dx$.

If $dU$ vanishes at the point, then $x=\sigma/2$, and
\[
dW=\Bigl(\frac{3\sigma^2}{4}-\sigma^2+\eta\Bigr)dx=-\frac{\sigma^2-4\eta}{4}\,dx\neq 0,
\]
because $\sigma^2-4\eta\neq 0$ and $q$ is odd.

At an affine point with $y=0$ we have $f(x)=0$ and $f'(x)\neq 0$, as $f$ is squarefree; differentiating $y^2=f(x)$ gives $dx=\bigl(2y/f'(x)\bigr)dy$, which vanishes at the point, so $y$ is a local parameter and
$dT=dy-\ell_1\,dx=dy\neq 0$.

At $P_\infty$, the valuations computed in Step~2
give $v_{P_\infty}(T/W)=1$ when $g=2$
and $v_{P_\infty}(W/T)=1$ when $g=3$.
In each case, the indicated affine coordinate
pulls back to a local parameter, so $d\Phi$
does not vanish at $P_\infty$.
For $g=1$ and for $g\geq4$, all three affine
coordinates of the corresponding chart vanish
to order at least two at $P_\infty$.
Hence $d\Phi$ vanishes there; the resulting
singularity is analysed in Step~8.

\emph{\review{Step 6.} the fibres of $\Phi$ and the smooth locus of $C$.}
Since $f(\alpha)\neq 0$, the four points $P_1,P_2,P_1^-,P_2^-$ are not Weierstrass points of the model \eqref{eq:model}. At these points $u=\varphi(x)$ and $w=x\varphi(x)$ vanish, while
$t(P_1)=\beta-\lambda(\alpha)=0$, $t(P_2)=\beta^q-\lambda(\alpha^q)=0$,
$t(P_1^-)=-2\beta$, and $t(P_2^-)=-2\beta^q$. 
Therefore
$\Phi(P_1)=\Phi(P_2)=(0:0:0:1)=:Q\in C(\F_q)$,
$\Phi(P_1^-)=(0:0:-2\beta:1)$, and
$\Phi(P_2^-)=(0:0:-2\beta^q:1)$.

If $\beta\in\F_q$, the last two coincide with the rational point $Q^-:=(0:0:-2\beta:1)$, distinct from $Q$ because $\beta\neq 0$; if $\beta\in\F_{q^2}\setminus\F_q$, they are distinct and conjugate, and form a closed point $\Phi(P^-)\in C$ of degree two. Moreover, $\Phi(P_\infty)$ is the unique point of $C$ lying on the hyperplane $\{S=0\}$. Combining this with Step~4, the only identifications made by $\Phi$ on geometric points are
$P_1,P_2\longmapsto Q$,
and, when $\beta\in\F_q,$ 
$P_1^-,P_2^-\longmapsto Q^-$.
In particular $\nu^{-1}(Q)=\{P\}$ and, when $\beta\in\F_q$, $\nu^{-1}(Q^-)=\{P^-\}$, so $C$ is unibranch at these points.

Let $V\subseteq\tC\otimes\overline{\F}_q$ be the complement of the identified points, namely
\[
\begin{gathered}
V=(\tC\otimes\overline{\F}_q)\setminus\{P_1,P_2\}
\quad\text{if }\beta\notin\F_q,\\
V=(\tC\otimes\overline{\F}_q)\setminus\{P_1,P_2,P_1^-,P_2^-\}
\quad\text{if }\beta\in\F_q ,
\end{gathered}
\]
further removing the point $P_\infty$ from $V$ when $g=1$ or $g\geq 4$.
The morphism $\Phi$ is finite and birational by Step~4, hence
is the normalization of $C$, proving (b).
Its restriction to $V$ is injective by the fibre computation
above and has nonzero differential by Step~5.
The Lemma of Shafarevich \cite[Chapter~II, \S5.4]{Sha13}
therefore identifies $V$ with its image, which is smooth.
Consequently, the only possible singularities are $Q$,
also $Q^-$ when $\beta\in\F_q$, and $Q_\infty$ when
$g=1$ or $g\geq4$.

Moreover, $Q$ is an ordinary node. Indeed, its two geometric branches are the images of the smooth points $P_1$ and $P_2$, and since $d\Phi$ is injective at $P_1$ and $P_2$ (Step~5), each branch is smooth with a well-defined tangent direction. On the chart $(u,w,t)$ we have, at $P_1$, $du=\varphi'(\alpha)\,dx$ and, using $\varphi(\alpha)=0$, $dw=\bigl(\varphi(\alpha)+\alpha\varphi'(\alpha)\bigr)dx=\alpha\varphi'(\alpha)\,dx$, with $\varphi'(\alpha)=\alpha-\alpha^q\neq 0$; hence the tangent direction of the branch through $\Phi(P_1)$ satisfies $dw/du=\alpha$, and likewise $dw/du=\alpha^q$ at $P_2$. Since $\alpha\neq\alpha^q$, the two tangent directions are distinct --- and conjugate, so that the tangent cone of $Q$ is irreducible over $\F_q$ --- and $Q$ is an ordinary node. When $\beta\in\F_q$, the same computation at $P_1^-$ and $P_2^-$ shows that $Q^-$ is an ordinary node as well.

\emph{Step 7. The local rings at $Q$ and $Q^-$.}
By Step~6, $Q$ is an ordinary node, hence
$\delta_Q=1$.
Since $Q$ is rational and its unique branch
is the degree-two point $P$, it is a non-split
singularity. Lemma~\ref{lem:localchar} therefore gives
$\OO_{C,Q}=\F_q+\mm_{\tC,P}$.
When $\beta\in\F_q$, the same argument applies
to $Q^-$ and $P^-$, giving
$\delta_{Q^-}=1$ and
$\OO_{C,Q^-}=\F_q+\mm_{\tC,P^-}$.

\emph{\review{Step 8.} the point $Q_\infty$ and the conclusion.}
Set $Q_\infty:=\Phi(P_\infty)$. By Step~6, $P_\infty$ is the unique point of $\tC\otimes\overline{\F}_q$ mapping to $Q_\infty$, so $C$ is unibranch at $Q_\infty$ with branch $P_\infty$, and $\deg Q_\infty=\deg P_\infty=1$. For $2\leq g\leq 3$ the point $Q_\infty$ is a smooth point of $C$ by Step~6, so $\Sigma_\infty=\emptyset$, as claimed in (e).

Assume now $g=1$ or $g\geq4$.
For $g=1$, the coordinates $U/W$, $T/W$ and $S/W$
have valuations $2$, $3$ and $6$ at $P_\infty$,
respectively, by Step~2.
They generate the maximal ideal of
$\OO_{C,Q_\infty}$, so every nonzero element
of this ideal has valuation at least two.
Thus $1\notin S(Q_\infty)$, while
$2,3\in S(Q_\infty)$, and consequently
$S(Q_\infty)=\langle2,3\rangle$.
Hence $Q_\infty$ is an ordinary cusp and
$\delta_{Q_\infty}=1$ by Lemma~\ref{lem:gaps}.
Moreover,
$\OO_{C,Q_\infty}\subseteq
\F_q+\mm_{\tC,P_\infty}^2$,
and equality holds because both subspaces
have codimension one in $\OO_{\tC,P_\infty}$.

For $g\geq4$, the coordinates $U/T$, $W/T$ and $S/T$
have valuations $2g-3$, $2g-5$ and $2g+1$,
respectively, and generate the maximal ideal
of $\OO_{C,Q_\infty}$.
The same argument shows that every nonzero
element of this ideal has valuation at least
$2g-5$.
Therefore $1,\ldots,2g-6\notin S(Q_\infty)$,
whereas
$\langle2g-5,2g-3,2g+1\rangle\subseteq S(Q_\infty)$.
Lemma~\ref{lem:gaps} gives
\[
2g-6\leq\delta_{Q_\infty}
\leq
\#\bigl(\mathbb N\setminus
\langle2g-5,2g-3,2g+1\rangle\bigr).
\]
In particular, $Q_\infty$ is singular.
This proves (e).

The arithmetic genus of $C$ is now computed by the genus formula $\pi=g+\sum_{Q'\in\Sing C}\delta_{Q'}$ (see \cite{AP96}), which gives the values stated in (c) and (d). Since $\nu^{-1}(Q)=\{P\}$ with $\deg P=2$, there is no rational branch over $Q$, i.e., $a_Q=0$, and likewise $a_{Q^-}=0$; over $Q_\infty$, when singular, the unique branch $P_\infty$ is rational, so $a_{Q_\infty}=1$ and $Q_\infty$ contributes no additional rational point. Equation~\eqref{lem:pointcount} then yields
$\#C(\F_q)=\#\tC(\F_q)+k$,
where $k$ is the number of nodes, as stated in (c) and (d).

\review{Finally, to prove (f), note that if} $2\leq g\leq 3$, then $\pi-g=k$, and if $\tC$ is maximal over $\F_q$, that is, $\#\tC(\F_q)=q+1+g\lfloor 2\sqrt{q}\rfloor$, then
\[
\#C(\F_q)=\#\tC(\F_q)+(\pi-g)=q+1+g\lfloor 2\sqrt{q}\rfloor+(\pi-g),
\]
which is the Aubry--Perret bound \eqref{eq:AP} over $\F_q$. For $g=1$ and for $g\geq 4$ we have instead $\pi-g=k+\delta_{Q_\infty}$, whence
\[
\#C(\F_q)=\#\tC(\F_q)+k\;\leq\;q+1+g\lfloor 2\sqrt{q}\rfloor+(\pi-g)-\delta_{Q_\infty},
\]
\review{with equality precisely when $\tC$ is maximal over $\F_q$. In that case the deficit from the Aubry--Perret bound is exactly $\delta_{Q_\infty}\geq 1$.} This proves (f) and completes the proof.
\end{proof}

\subsection{\review{A modified projection in the elliptic case}}\label{subsec:elliptic}

By Theorem~\ref{thm:construction}(e), for $g=1$ the projected curve has an ordinary cusp at $Q_\infty$ besides the nodes and, by (f), it never attains the Aubry--Perret bound. We now modify the projection, replacing the coordinate $w=x\varphi(x)$, of pole order $6$ at $P_\infty$, by a function of pole order $5$, and show that this removes the singularity at infinity. Throughout this subsection $g=1$, so $f$ has degree $3$ and $h(x)=h_1x+h_0$ has degree one. Set $m_1:=\ell_0+\ell_1\sigma$ and define
\[
z:=x\,t+\ell_1u=xy-m_1x+\ell_1\eta\;\in K ,
\]
the second equality following from $\lambda(x)=\ell_1x+\ell_0$ and $\varphi(x)=x^2-\sigma x+\eta$. From the first expression, $z$ vanishes wherever $t$ and $u$ do; from the second, $v_{P_\infty}(z)=v_{P_\infty}(xy)=-5$. In fact $m_1x-\ell_1\eta$ interpolates the values of $xy$ at the two points \review{lying under $P$, since} $m_1\alpha-\ell_1\eta=\ell_0\alpha+\ell_1(\sigma\alpha-\eta)=\alpha(\ell_0+\ell_1\alpha)=\alpha\beta$, and likewise at $(\alpha^q,\beta^q)$.

The new coordinates belong to $L(5P_\infty)$, whose basis $1,x,x^2,y,xy$ gives an embedding into $\PP^4$. Consequently, the map below is again a linear projection of a projective embedding, now with coordinates of pole orders $4,5,3,0$. The order-five coordinate makes the image smooth at infinity, while the local modification at $P$ is retained. This supplies an explicit nodal realization of the intrinsic degree-two-point construction in the elliptic case.

\begin{theorem}\label{thm:elliptic}
Let $g=1$, keep the standing hypotheses on $P$, and consider the map
\[
\Phi_{\mathrm e}\colon\tC\dashrightarrow\PP^3,
\qquad
(x,y)\longmapsto\bigl(\varphi(x):z:y-\lambda(x):1\bigr)=:(U:W:T:S).
\]
Then:
\begin{enumerate}
\item[(a)] $\Phi_{\mathrm e}$ extends to a morphism on all of $\tC$, with $\Phi_{\mathrm e}(P_\infty)=(0:1:0:0)$, and $\Phi_{\mathrm e}\colon\tC\to C_{\mathrm e}:=\Phi_{\mathrm e}(\tC)$ is the normalization of its image;
\item[(b)] $\Sing C_{\mathrm e}=\{Q\}$, where $Q=(0:0:0:1)$ is a non-split singularity whose branch is $P$ --- irrespective of the position of $\beta$; consequently the arithmetic genus of $C_{\mathrm e}$ is $2$ and $\#C_{\mathrm e}(\F_q)=\#\tC(\F_q)+1$;
\item[(c)] if $\tC$ is maximal over $\F_q$, then $C_{\mathrm e}$ attains the Aubry--Perret bound \eqref{eq:AP} over $\F_q$;
\end{enumerate}
\end{theorem}

\begin{proof}
Since $xt=z-\ell_1u$ by definition of $z$, substituting $x=(z-\ell_1u)/t$ into $t^2\varphi(x)=t^2u$ and into the functional equation \eqref{eq:functional} multiplied by $t$ yields the identities of functions on $\tC$
\begin{equation}\label{eq:Ie}
t^2u=(z-\ell_1u)^2-\sigma\,t\,(z-\ell_1u)+\eta\,t^2
\tag{$\mathrm{I_e}$}
\end{equation}
\begin{equation}\label{eq:IIe}
t^3+2\ell_1t\,(z-\ell_1u)+2\ell_0t^2=u\bigl(h_1(z-\ell_1u)+h_0t\bigr);
\tag{$\mathrm{II_e}$}
\end{equation}
moreover $x=(z-\ell_1u)/t$ shows that $\F_q(u,z,t)=\F_q(x,y)=K$, so $\Phi_{\mathrm e}$ is birational onto its image.  As before, $C_{\mathrm e}$ is defined as this image. Relations \eqref{eq:Ie} and \eqref{eq:IIe} hold on its affine chart, but are not asserted to generate its defining ideal; their common zero locus also contains the line $t=0$, $z=\ell_1u$.

\emph{Extension at $P_\infty$.} The pole orders of $(u,z,t)$ at $P_\infty$ are $4$, $5$ and $3$. Dividing the tuple $(u:z:t:1)$ by $z$ we obtain the coordinates $u/z$, $t/z$ and $1/z$ of the chart $\{W\neq 0\}$, of valuations $1$, $2$ and $5$ at $P_\infty$. Hence $\Phi_{\mathrm e}$ extends to $P_\infty$ with $\Phi_{\mathrm e}(P_\infty)=(0:1:0:0)$ and, since $u/z$ has valuation~$1$, the differential $d\Phi_{\mathrm e}$ does not vanish at $P_\infty$.

\emph{Identifications.} Let $R,R'$ be distinct points of $\tC\otimes\overline{\F}_q$, away from $P_\infty$, with $\Phi_{\mathrm e}(R)=\Phi_{\mathrm e}(R')$, say $R=(x_0,y_0)$ and $R'=(x_0',y_0')$. From $u(R)=u(R')$ we get $x_0'\in\{x_0,\sigma-x_0\}$; if $x_0'=x_0$, then $t(R)=t(R')$ forces $y_0'=y_0$, a contradiction; so $x_0'=\sigma-x_0\neq x_0$. Then $t(R')=t(R)$ gives $y_0'=y_0+\ell_1(\sigma-2x_0)$, and $z(R')=z(R)$ reduces, by $z=xt+\ell_1u$ and $t(R')=t(R)$, to $t(R)\,(x_0'-x_0)=0$, that is, $t(R)=0$. Hence $y_0=\lambda(x_0)$ and $y_0'=\lambda(x_0')$, and the functional equation gives $\varphi(x_0)h(x_0)=0=\varphi(x_0')h(x_0')$. If $\varphi(x_0)=0$, then $\{x_0,x_0'\}=\{\alpha,\alpha^q\}$ and $\{R,R'\}=\{P_1,P_2\}$. Otherwise $h(x_0)=h(\sigma-x_0)=0$ with $x_0\neq\sigma-x_0$, which is impossible since $h$ has degree one. Therefore $\Phi_{\mathrm e}$ identifies exactly $P_1$ with $P_2$ and is injective elsewhere; in particular $\Phi_{\mathrm e}(P_1)=\Phi_{\mathrm e}(P_2)=(0:0:0:1)=Q$, while $\Phi_{\mathrm e}(P_1^-)=(0:-2\alpha\beta:-2\beta:1)$ and $\Phi_{\mathrm e}(P_2^-)=(0:-2\alpha^q\beta^q:-2\beta^q:1)$ remain distinct --- also when $\beta\in\F_q$, since $\alpha\neq\alpha^q$.

\emph{The differential at the affine points.} We have $du=\varphi'(x)\,dx$, $dt=dy-\ell_1\,dx$ and $dz=t\,dx+x\,dt+\ell_1\,du$. At a point with $y=0$ the function $y$ is a local parameter, $dx$ vanishes there and $dt=dy\neq 0$. At a geometric point $R=(x_0,y_0)$ with $y_0\neq0$, the function $x-x_0$ is a local parameter; if $du$ and $dt$ vanish simultaneously, then $x=\sigma/2$ and $dz=t\,dx$ there, and if also $dz=0$, then $t=0$, i.e., $y=\lambda(\sigma/2)$; from $t=0$ and $\varphi(\sigma/2)=-(\sigma^2-4\eta)/4\neq 0$ the functional equation gives $h(\sigma/2)=0$, whence
\[
0=\bigl(f-\lambda^2\bigr)'(\sigma/2)=(\varphi h)'(\sigma/2)=\varphi(\sigma/2)\,h_1\neq 0,
\]
a contradiction (the first equality is $dt=0$ rewritten as $f'(\sigma/2)=2y\ell_1$). Hence $d\Phi_{\mathrm e}$ vanishes nowhere on $\tC\otimes\overline{\F}_q$.

\emph{Conclusion.} As in Steps~6 of the proof of Theorem~\ref{thm:construction}, the Lemma of Shafarevich \cite[Chapter~II, \S5.4]{Sha13} applied to $V:=(\tC\otimes\overline{\F}_q)\setminus\{P_1,P_2\}$ shows that $C_{\mathrm e}$ is smooth away from $Q$ and that $\Phi_{\mathrm e}$ is the normalization of $C_{\mathrm e}$, proving (a). At $Q$ the two branches are smooth with tangent directions determined by $(dz-\ell_1du)/dt=x$, evaluated at $P_1$ and $P_2$; here $dt(P_1)\neq 0$ because $dt(P_1)=0$ would give $(\varphi h)'(\alpha)=\varphi'(\alpha)h(\alpha)=0$, while $\varphi'(\alpha)=\alpha-\alpha^q\neq 0$ and $h(\alpha)\neq 0$, the root of $h$ being rational. The two directions are therefore distinct and conjugate, and $Q$ is an ordinary node. Thus $\delta_Q=1$.
Since $Q$ is rational and its unique branch
is the degree-two point $P$, it is a non-split
singularity, and Lemma~\ref{lem:localchar} gives
$\OO_{C_{\mathrm e},Q}=\F_q+\mm_{\tC,P}$.
The genus formula and \eqref{lem:pointcount}
yield $p_a(C_{\mathrm e})=g+1=2$ and
$\#C_{\mathrm e}(\F_q)=\#\tC(\F_q)+1$,
proving (b).
Assertion (c) follows from these equalities
when $\tC$ is maximal.
\end{proof}

\subsection{\review{Maximal curves on cones via the Rosa--St\"ohr model}}\label{subsec:cone}

For $g\geq4$, we use a canonical model on a cone in $\PP^g$ to retain one non-split singularity and avoid the additional singularity at infinity. The geometric model comes from Rosa--St\"ohr \cite{RS02}; our task is to express it in the interpolation data over $\F_q$, identify the Frobenius-conjugate branches at the vertex, and deduce maximality over the ground field. The key observation is that the functional equation \eqref{eq:functional} is, after the substitution
\[
Y:=\frac{t}{\varphi(x)}=\frac{y-\lambda(x)}{\varphi(x)}\;\in K ,
\]
\review{a Rosa--St\"ohr normal form. Dividing} \eqref{eq:functional} by $\varphi(x)^2$ gives the plane model
\begin{equation}\label{eq:RSform}
\varphi(x)\,Y^2+2\lambda(x)\,Y-h(x)=0
\end{equation}
of the function field $K$, which is an equation $c_2(x)Y^2+c_1(x)Y+c_0(x)=0$ as in \cite[Theorem~1]{RS02}, with
$c_2=\varphi$, $c_1=2\lambda$, $c_0=-h$, 
of degrees $2$, at most $1$ and $2g-1$; setting $\pi:=g+1$, these satisfy $\deg c_2\leq 2$, $\deg c_1\leq\pi$ and $\deg c_0\leq 2\pi-2$, with equality attained by $c_2$. The polynomial in \eqref{eq:RSform} is absolutely
irreducible, since the substitution
$y=\varphi(x)Y+\lambda(x)$ identifies its locus
$\varphi(x)\neq0$ with an open subset of $\tC$,
and $\gcd(\varphi,\lambda)=1$ excludes components
supported on $\varphi(x)=0$.
\review{To compare the notation, recall that the letter} $g$ in \cite{RS02} denotes the \emph{arithmetic} genus of the trigonal curve, which in our setting is $\pi=g+1$; thus our hypothesis $g\geq 4$ is exactly the standing hypothesis $\pi\geq 5$ of \cite{RS02}, and the hyperelliptic Gorenstein curve of arithmetic genus $\pi-1$ appearing in \cite[Theorem~3]{RS02} has, in our setting, arithmetic genus $g$ --- it will be $\tC$ itself. Consider the rational map
\[
\Psi\colon\tC\longrightarrow\PP^{g},
\qquad
(x,y)\longmapsto\bigl(1:x:x^2:\cdots:x^{g-1}:Y\bigr),
\]
which extends uniquely to a morphism on the smooth complete curve $\tC$. Let $C'$ be its image. By \cite[Theorem~1]{RS02}, applied over $\overline{\F}_q$ using the irreducibility and degree conditions just checked, $C'$ is a canonical Gorenstein curve of arithmetic genus $\pi=g+1$ on the cone $S\subset\PP^g$ over the rational normal curve of degree $g-1$, with vertex $Q'=(0:\cdots:0:1)$. All coordinate functions are defined over $\F_q$, so $C'$, $S$ and $Q'$ are defined over $\F_q$. We use the convention for the Maroni invariant \review{in \cite[Theorem~2.1]{RS02}, where it is} the smaller integer $m$ in the pair $m\leq n$, $m+n=\pi-2$; the cone corresponds to $m=0$.

We call a curve trigonal over $\F_q$ if its gonality
over $\F_q$ is three. Here gonality is understood as
the minimum degree of a divisor, in the sense of
St\"ohr, admitting two linearly independent global
sections over $\F_q$; a $g^1_3$ denotes a linear
system of degree three and dimension one in this
sense. The precise definitions, including the
possibility of non-removable base points, are
recalled in Subsection~\ref{subsec:linsys}.

\begin{theorem}\label{thm:cone}
Let $g\geq 4$ and keep the standing hypotheses on $P$. Then:
\begin{enumerate}
\item[(a)] $\Psi\colon\tC\to C'$ is the normalization of $C'$, and it is an isomorphism away from the vertex $Q'$ of the cone;
\item[(b)] $\Sing C'=\{Q'\}$, and $Q'$ is a non-split singularity of $C'$ whose branch is the hyperelliptic conjugate $P^-$ of $P$, the closed point of degree two lying under $P_1^-=(\alpha,-\beta)$ and $P_2^-=(\alpha^q,-\beta^q)$; consequently the arithmetic genus of $C'$ is $g+1$ and $\#C'(\F_q)=\#\tC(\F_q)+1$;
\item[(c)] if $\tC$ is maximal over $\F_q$, then $C'$ attains the Aubry--Perret bound \eqref{eq:AP} over $\F_q$ --- for every $g\geq 4$;
\item[(d)] $C'$ is a trigonal Gorenstein curve with Maroni invariant zero, whose $g^1_3$ is cut out by the rulings of the cone and has $Q'$ as its unique, non-removable base point, in the terminology recalled in Subsection~\ref{subsec:linsys}.
\end{enumerate}
\end{theorem}

\begin{proof}
Since $y=\varphi(x)Y+\lambda(x)$, we have
$\F_q(x,Y)=K$. Thus $\Psi$ is birational onto $C'$
and, being a nonconstant morphism of complete curves
with smooth source, is its normalization.

From the coordinates of $\Psi$, a point $R\in\tC$
maps to $Q'$ if and only if
\[
v_R(Y)<\min\{0,(g-1)v_R(x)\}.
\]
The identity $Y=h/(y+\lambda)$ gives
$v_P(Y)=v_P(h(x))\geq0$, since $y+\lambda$
has residue $2\beta\neq0$ at $P$.
At $P^-$ we have $v_{P^-}(Y)=-1$,
while $x$ is regular, so the criterion holds.
At $P_\infty$ it fails, since
$v_{P_\infty}(Y)=-(2g-3)>-2(g-1)$.
At every other point, both $x$ and $Y$ are
regular, so it fails there as well.
Thus $\Psi^{-1}(Q')=\{P^-\}$.

As established above using \cite[Theorem~1]{RS02},
$C'$ is canonical and Gorenstein of arithmetic
genus $g+1$. Since $\Psi$ is the normalization and
$\Psi^{-1}(Q')=\{P^-\}$ with
$\deg P^-=2\neq1=\deg Q'$, the point $Q'$
is singular. The genus formula
now gives
$\sum_{R\in\Sing C'}\delta_R=(g+1)-g=1$.
Consequently, $\Sing C'=\{Q'\}$ and $\delta_{Q'}=1$,
so $Q'$ is non-split and $\Psi$ is an isomorphism
away from $Q'$. Moreover,
$\OO_{C',Q'}=\F_q+\mm_{\tC,P^-}$ by
Lemma~\ref{lem:localchar}.
This proves (a) and the singularity assertions
in (b).

Since the unique branch over $Q'$ has degree two,
\eqref{lem:pointcount} gives
$\#C'(\F_q)=\#\tC(\F_q)+1$, completing (b).
If $\tC$ is maximal, this equality and
$p_a(C')-g=1$ give (c).

Finally, the geometric assertions in (d) follow
from \cite[Theorem~2]{RS02}, with Maroni invariant
zero under the convention recalled above.
The assertion $\gon_{\F_q}(C')=3$ is proved in
Theorem~\ref{thm:gonality}(b).
\end{proof}

\begin{remark}\label{rem:coneswap}
Replacing $\beta$ by $-\beta$, i.e., choosing the other point of $\tC\otimes\F_{q^2}$ over $\alpha$, exchanges $P$ and $P^-$ and replaces $\lambda$ by $-\lambda$ in \eqref{eq:RSform}; so a cone model whose node lies under $P$ is obtained from the normal form $\varphi Y^2-2\lambda Y-h=0$. Note also that the construction is insensitive to the \review{position of $\beta$. As in} Theorem~\ref{thm:elliptic}, the case distinction in Theorem~\ref{thm:construction}(c)--(d) disappears, and the cone model always has a single node. The restriction $g\geq 4$ in Theorem~\ref{thm:cone} reflects the standing hypothesis $\pi\geq 5$ on the arithmetic genus in \cite{RS02}.
\end{remark}

\subsection{Families and explicit examples}\label{subsec:examples}

Theorem~\ref{thm:construction}(f), Theorem~\ref{thm:elliptic} and Theorem~\ref{thm:cone} convert any family of maximal elliptic or hyperelliptic curves with a model \eqref{eq:model} into families of singular curves attaining the Aubry--Perret bound, provided a closed point $P$ satisfying the standing hypotheses $\alpha\notin\F_q$ and $f(\alpha)\neq0$ is chosen. Observe that maximality is required \review{over the \emph{ground} field. To apply} the constructions to a curve which is maximal over $\F_{q^2}$, one takes $\F_{q^2}$ as the ground field and $\alpha\in\F_{q^4}\setminus\F_{q^2}$. 

Two classical families provide input curves. \review{By Tafazolian \cite{Taf12}, the hyperelliptic curves}
$y^2=x^m+1$
are maximal over $\F_{q^2}$ whenever $m$ divides $q+1$. The second is the Artin--Schreier family
$y^2=x^q+x$, 
which is maximal over $\F_{q^2}$, being a quotient of the Hermitian curve; see Tafazolian--Torres \cite[Theorem~3.3]{TT17}.

\begin{remark}\label{rem:admissiblepoints} 
The constructions are conditional on the choice of a point $P$
as in Subsection~\ref{subsec:degtwo}. A sufficient criterion for
the existence of such a point follows by counting closed points
of degree two. Their number is
\[
B_2(\tC)
=\frac{\#\tC(\F_{q^2})-\#\tC(\F_q)}{2},
\]
since each closed point of degree two gives two
Frobenius-conjugate points over $\F_{q^2}$.

We exclude two types of points. First, for each $a\in\F_q$,
the fibre of $x$ above $a$ contains a closed point of degree
two exactly when $f(a)$ is a non-square in $\F_q$.
Thus at most $q$ closed points of degree two have rational
$x$-coordinate; the unique point above infinity is rational.
Second, the closed points of degree two with $f(x(P))=0$
have $y(P)=0$ and correspond precisely to the irreducible
quadratic factors of $f$. Their number is at most
$\lfloor(2g+1)/2\rfloor=g$.
Consequently,
\[
\#\{P\in\tC:\deg P=2,\ x(P)\notin\F_q,\
                   f(x(P))\neq0\}
\geq B_2(\tC)-q-g.
\]
In particular, $B_2(\tC)>q+g$ guarantees a point satisfying
the standing hypotheses.

If $q=s^2$ and $\tC$ is maximal over $\F_q$, its Frobenius
eigenvalues are all $-s$. The point-count formula therefore gives
$\#\tC(\F_{s^2})=s^2+1+2gs$, 
$\#\tC(\F_{s^4})=s^4+1-2gs^2$,
and hence
\[
B_2(\tC)=\frac{s^4-(2g+1)s^2-2gs}{2}.
\]
Here degree two is understood over the ground field
$\F_{s^2}$. \review{The criterion is sufficient but not necessary, since
when} it fails, admissible points may still exist.
It does not distinguish between $\beta\in\F_q$ and
$\beta\notin\F_q$.\end{remark}

\begin{proposition}\label{cor:infinitefamilies}
Fix an integer $g\geq4$. For every prime
$p\equiv-1\pmod{2g+1}$, there exists a maximal canonical
Gorenstein curve $C'_p\subset\PP^g$ over $\F_{p^2}$
whose normalization is the smooth curve
$\tC_p:y^2=x^{2g+1}+1$ and whose only singularity is
a non-split singularity. Moreover,
$p_g(C'_p)=g$, 
$p_a(C'_p)=g+1$, 
$\#C'_p(\F_{p^2})=p^2+2gp+2$,
and $\gon_{\F_{p^2}}(C'_p)=3$.
In particular, such curves exist over infinitely many
distinct finite fields.\end{proposition}
\begin{proof}
Put $m=2g+1$. Since $p\equiv-1\pmod m$ and $m\geq9$,
the prime $p$ is odd and $m\mid p+1$. \review{The result of Tafazolian \cite{Taf12} recalled above shows that $\tC_p$ is maximal over $\F_{p^2}$.} Also $p\nmid m$, so
$x^m+1$ is squarefree and $\tC_p$ has genus $g$.

By Remark~\ref{rem:admissiblepoints}, the number of
closed points of degree two over $\F_{p^2}$ is
\[
B_2(\tC_p)=\frac{p^4-(2g+1)p^2-2gp}{2}.
\]
We claim that $B_2(\tC_p)>p^2+g$. Indeed, the congruence
$p\equiv-1\pmod{2g+1}$ implies $p\geq2g$, and hence
\[
p^2-p-3-2g
\geq 4g^2-4g-3
=(2g-3)(2g+1)>0.
\]

Consequently,
\[
\begin{aligned}
2\bigl(B_2(\tC_p)-p^2-g\bigr)
&=p^4-3p^2-2g(p^2+p+1)\\
&=(p^2+p+1)(p^2-p-3-2g)+4p+3
>0.
\end{aligned}
\]
Remark~\ref{rem:admissiblepoints} therefore guarantees
a closed point $P$ of degree two over $\F_{p^2}$ with
$x(P)\in\F_{p^4}\setminus\F_{p^2}$ and
$x(P)^{2g+1}+1\neq0$.

Theorem~\ref{thm:cone} supplies the asserted canonical
Gorenstein curve, its unique non-split singularity, its genera
and its point count. Its gonality over $\F_{p^2}$ is
three by Theorem~\ref{thm:gonality}(b).

Finally, since $\gcd(-1,2g+1)=1$, Dirichlet's theorem
gives infinitely many primes in this congruence class.
Thus the construction yields curves over infinitely
many distinct finite fields; no assertion of infinitely
many classes over a fixed finite field is intended.
\end{proof}

In the examples below the ground fields are $\F_{25}=\F_5(a)$, with $a^2+a+1=0$, for the elliptic case, and $\F_{81}=\F_3(\theta)$, with $\theta^4=\theta^3+1$, for genus two; the third example is stated over $\F_3$. \review{In each example the curve is the image of the morphism in the corresponding construction theorem.} The interpolation data can be checked by $f-\lambda^2=\varphi h$; the singularities and point counts then follow from the preceding theorems and the stated properties of the normalization.

\begin{example}[$g=1$, one non-split singularity, maximal]\label{ex:g1max}
Let $\tC\colon y^2=x^3+1$ over $\F_{25}$, which is maximal with $\#\tC(\F_{25})=36=25+1+\lfloor2\sqrt{25}\rfloor$. The element $n:=1+2a$ is a non-square of $\F_{25}$; let $b\in\F_{625}$ with $b^2=n$ and take $\alpha:=nb$, so that $\alpha^2=n^3=2+4a$ and $\alpha\notin\F_{25}$. Then $f(\alpha)=n^4b+1=4b+1\neq0$ is a square in $\F_{625}$ whose square roots $\pm\beta$ do not lie in $\F_{25}$, and $P$ is the closed point of degree two under $(\alpha,\beta)$. 


Choose the sign of $\beta$ so that
$\beta=(4a+1)\alpha+a$.
The interpolation data are
$\varphi(x)=x^2+a+3$,
$\lambda(x)=(4a+1)x+a$ and $h(x)=x+3a$,
as verified by
$x^3+1-\lambda(x)^2=\varphi(x)h(x)$.

\review{By Theorem~\ref{thm:elliptic}, $C_{\mathrm e}$ is maximal of arithmetic genus $2$, with $\#C_{\mathrm e}(\F_{25})=\#\tC(\F_{25})+1=37$.}
\end{example}

\begin{example}[$g=1$, non-split and cusp]\label{ex:g1cusp}
With the same $\tC$ and $P$ of Example~\ref{ex:g1max}, the projection in Theorem~\ref{thm:construction} produces instead a curve $C\subset\PP^3$ with two singular points: the non-split singularity $Q=(0:0:0:1)$ and the ordinary cusp $Q_\infty=(0:1:0:0)$, so that $\pi=3$. 

By Theorem~\ref{thm:construction},
\[
\#C(\F_{25})=37=25+1+1\cdot\lfloor 2\sqrt{25}\rfloor+(\pi-g)-\delta_{Q_\infty}<38.
\]
\review{The geometry of $C$ remains interesting. As shown} in Section~\ref{sec:gonality}, the cusp raises the gonality from $2$ to $3$.
\end{example}

\begin{example}[$g=2$, one non-split singularity, maximal]\label{ex:g2max}
Let $\tC\colon y^2=x^5+1$ over $\F_{81}$, which is maximal with $\#\tC(\F_{81})=118=81+1+2\cdot\lfloor2\sqrt{81}\rfloor$, in accordance with \cite{Taf12}, since $5$ divides $q+1=10$ for $q=9$. Take for $P$ the closed point of degree two whose $x$-coordinates are the roots of
$\varphi(x)=x^2-\theta^{68}x+\theta^{30}$,
that is, $\sigma=\theta^{68}$ and $\eta=\theta^{30}$; then $f(\alpha)$ is a square in $\F_{81^2}$ with $\beta\notin\F_{81}$, and the interpolation data are
$\lambda(x)=\theta^{50}x+\theta^{70}$, 
$h(x)=x^3+\theta^{68}x^2+\theta^{51}x+\theta^{40}$.

Here $2\leq g\leq 3$, so Theorem~\ref{thm:construction} \review{shows that $C$ has} a single non-split singularity, $\pi=3$, 
$\#C(\F_{81})=119=81+1+2\cdot\lfloor 2\sqrt{81}\rfloor+1$,
and $C$ is a maximal curve over $\F_{81}$; it is trigonal, as shown in Section~\ref{sec:gonality}.
\end{example}

\begin{example}[$g=2$, two non-split singularities]\label{ex:g2twonodes}
Let $\tC\colon y^2=x^5-x+1$ over $\F_3$ and take $\alpha$ with $\alpha^2=-1$, so that $\varphi(x)=x^2+1$; then $f(\alpha)=1$ and $\beta=1\in\F_3$, and we are in Theorem~\ref{thm:construction}, \review{case (d), so the curve $C$ has} the two non-split singularities $Q=(0:0:0:1)$ and $Q^-=(0:0:1:1)$, with
$\sigma=0$,
$\eta=1$,
$\lambda(x)=1$,
$h(x)=x^3-x$,
and $\pi=g+2=4$. 


Here $\tC$ is not maximal over $\F_3$ --- $\#\tC(\F_3)=7$ --- \review{and $\#C(\F_3)=9$, so no maximality is claimed.} This curve reappears in Section~\ref{sec:gonality}, where its gonality is shown to be $3$; a companion curve over $\F_3$, tetragonal, is treated there as well.
\end{example}

\begin{remark}\label{rem:cased}
It is natural to ask whether case (d) of Theorem~\ref{thm:construction} can occur with $\tC$ maximal over the ground field. For $\tC:y^2=x^3+1$ over $\F_{25}$, \review{it cannot. Indeed, if} $\beta\in\F_{25}$, then $\alpha^3=\beta^2-1\in\F_{25}$. Since $\alpha\notin\F_{25}$, we have $\alpha\neq0$, and hence $\alpha^{72}=1$. But the group of $72$nd roots of unity in $\F_{625}^{\times}$ has order $\gcd(72,624)=24$ and equals $\F_{25}^{\times}$, contradicting $\alpha\notin\F_{25}$. We do not know if such input data can occur for a maximal normalization in general.

\end{remark}

\section{Gonality of curves}\label{sec:gonality}

We now determine the gonality over $\F_q$
of the curves constructed in
Section~\ref{sec:construction}. We first recall
the divisor language and the distinction
between free and base-point-free linear
systems. We then derive the local degree
formula used in the gonality computations.
Throughout, divisors and spaces of sections
are taken over the stated ground field.

\subsection{Divisors, linear systems and gonality}\label{subsec:linsys}

Let $X$ be a complete, geometrically irreducible curve over a field $k$, not necessarily algebraically closed, with normalization $\tX$. A \emph{divisor} on $X$ --- equivalently, a coherent fractional ideal sheaf on $X$ --- is a product
$D=\prod_{Q\in X}D_Q$,
where each $D_Q\subset k(X)$ is a fractional ideal of the local ring $\OO_{X,Q}$ and $D_Q=\OO_{X,Q}$ for almost all closed points $Q$. The structure sheaf is $\OO=\prod_Q\OO_{X,Q}$. 

Given divisors $D$ and $E$ we set $D\cdot E:=\prod_QD_QE_Q$, and we write $D\geq E$ when $D_Q\supseteq E_Q$ for every $Q$. The \emph{principal divisor} of $h\in k(X)^{\times}$ is $\divi(h):=\prod_Qh^{-1}\OO_{X,Q}$, two divisors being \emph{linearly equivalent}, $D\sim E$, when $D=E\cdot\divi(h)$ for some $h$; and
$H^0(D):=\bigcap_{Q\in X}D_Q$
is the $k$-vector space of global sections of $D$. 

The \emph{local degree} at $Q$ is normalized by $\deg_Q\OO_{X,Q}=0$ and determined by $\deg_QD_Q-\deg_QE_Q=\dim_k\bigl(D_Q/E_Q\bigr)$ whenever $D_Q\supseteq E_Q$; the \emph{degree} of $D$ is the finite sum $\deg_XD:=\sum_Q\deg_QD_Q$. Exactly as for smooth curves, principal divisors have degree zero, so that linearly equivalent divisors have equal degrees.

A \emph{linear system} of dimension $r$ and degree $d$ on $X$, a $g^r_d$, is a family
\[
\mathcal L=\bigl\{\divi(s)\cdot D\;:\;s\in V\setminus\{0\}\bigr\},
\]
where $D$ is a divisor with $\deg_XD=d$ and $V\subseteq H^0(D)$ is a $k$-subspace with $\dim_kV=r+1$; note that $\divi(s)\cdot D=\prod_Qs^{-1}D_Q\geq\OO$ for every $s\in H^0(D)\setminus\{0\}$. If $s_0,\dots,s_r$ is a basis of $V$, the divisor
\[
E:=\prod_{Q\in X}\bigl(s_0\OO_{X,Q}+\dots+s_r\OO_{X,Q}\bigr)
\]
satisfies $E\leq D$ and is minimal with the property $V\subseteq H^0(E)$ (cf.\ \cite[p.~156]{Car99}\review{). Following} \cite[Section~2, p.~190]{RS02}, we call $\mathcal L$ \emph{free} when $E=D$. 

A closed point $Q$ is a \emph{base point} of $\mathcal L$ if
$\OO_{X,Q}\subsetneq s^{-1}D_Q$
for every $s\in V\setminus\{0\}$,
that is, if it belongs to the support of every divisor in the system. The system is \emph{base-point-free} if it has no such point. 
Replacing $D$ by $E$ gives its associated free system
\[
\mathcal L_{\mathrm{free}}=\{\divi(s)\cdot E:s\in V\setminus\{0\}\}.
\]
A base point of $\mathcal L$ is \emph{non-removable} if it remains a base point of $\mathcal L_{\mathrm{free}}$, and \emph{removable} otherwise. Equivalently, $Q$ is a non-removable base point if \review{$E_Q$ is not principal. If} $E_Q=a\OO_{X,Q}$, then the generators $s_i/a$ generate the unit ideal, so at least one is a unit and the corresponding member avoids $Q$; the converse follows directly from the definition. Thus non-removable base points are necessarily singular. These definitions distinguish being free from being base-point-free and apply over the ground field $k$.

The \emph{gonality} of $X$ is
\[
\gon(X):=\min\bigl\{d\;:\;\text{there exists a }g^1_d\text{ on }X\bigr\}.
\]


\begin{remark}\label{rem:closedfield}
Over any ground field $k$, an invertible sheaf
generated by a two-dimensional $k$-space of
global sections defines a morphism
$X\to\PP^1_k$, whose degree equals the degree
of the sheaf. Thus this correspondence also
holds over finite fields. The distinction
relevant here comes from allowing non-invertible
divisors, whose pencils may have non-removable
base points. Consequently, for singular curves,
the gonality defined above need not equal the
minimum degree of a morphism to $\PP^1_k$.
\end{remark}

\begin{lemma}\label{lem:normalized}
Let $d=\gon(X)$. Then there exists $f\in k(X)\setminus k$ such that the pencil
\[
\bigl\{\divi(s)\cdot\OO\langle 1,f\rangle\;:\;s\in (k+kf) \setminus \{0\}\bigr\},
\] induced by 
\[
\OO\langle 1,f\rangle:=\prod_{Q\in X}\bigl(\OO_{X,Q}+f\OO_{X,Q}\bigr),
\] 
is a $g^1_d$. Consequently,
\begin{equation}\label{eq:gonmin}
\gon(X)=\min\bigl\{\deg_X\OO\langle 1,f\rangle\;:\;f\in k(X)\setminus k\bigr\}.
\end{equation}
\end{lemma}

\begin{proof}
Let $\{\divi(s)\cdot D:s\in V\setminus\{0\}\}$ be a $g^1_d$ and let $g,h$ be a basis of $V$. The divisor $E=\prod_Q(g\OO_{X,Q}+h\OO_{X,Q})$ satisfies $E\leq D$, hence $\deg_XE\leq d$, and $V\subseteq H^0(E)$, so that $\{\divi(s)\cdot E\}$ is a linear system of dimension $1$ and degree $\deg_XE$. By the minimality in the definition of the gonality, $\deg_XE=d$.

Take $f:=h/g$, which is not in $k$ since $g,h$ are $k$-linearly independent. From $g\OO_{X,Q}+h\OO_{X,Q}=g\,(\OO_{X,Q}+f\OO_{X,Q})$ for every $Q$ we get
$\OO\langle 1,f\rangle=E\cdot\divi(g)\sim E$,
so $\deg_X\OO\langle1,f\rangle=d$, and the substitution $s\mapsto gs$ identifies the pencil $\{\divi(s)\cdot\OO\langle1,f\rangle:s\in (k+kf)\setminus \{0\}\}$ with the $g^1_d$ associated to $E$. 
As for \eqref{eq:gonmin}, this is clear from the first part.
\end{proof}

The next result computes the right-hand side of \eqref{eq:gonmin} for the curves in this paper.

\begin{proposition}\label{prop:degformula2}
\label{prop:degformula}
Let $X$ be a curve over $\F_q$ whose singularities are either non-split or rational unibranch with rational branch, as in Definition~\ref{def:semigroup}. For $Q\in\Sing X$ let $P\in\tX$ be the branch of $X$ at $Q$. Then for every $f\in\F_q(X)\setminus\F_q$ one has
\begin{equation}
\label{eq:degformula}
\deg_X\OO\langle 1,f\rangle
=\deg_{\tX}\bigl(\divi_\infty(f)\bigr)
+\sum_{Q\in\Sing X}\varepsilon_Q(f),
\end{equation}
where $\divi_\infty(f)$ is the pole divisor of $f$ on $\tX$, 
\begin{equation*}
\varepsilon_Q(f)=
\begin{cases}
1 & \text{if } f\in\OO_{\tX,P}\setminus\OO_{X,Q},\\[2pt]
0 & \text{otherwise},
\end{cases}
\end{equation*}
if $Q$ is non-split, 
while
\[
\varepsilon_Q(f):=\deg_Q\bigl(\OO_{X,Q}+f\OO_{X,Q}\bigr)-\max\bigl(0,-v_P(f)\bigr),
\qquad
0 \leq \varepsilon_Q(f) \leq\delta_Q,
\] 
if $Q$ is rational unibranch with rational branch. 

In particular $\deg_{\tX}(\divi_\infty(f))\leq\deg_X\OO\langle1,f\rangle$, in accordance with \cite[Theorem~2.2]{Car99}.
\end{proposition}

\begin{proof}
We compute $\deg_Q\bigl(\OO_{X,Q}+f\OO_{X,Q}\bigr)$ at every point $Q\in X$.

Suppose first that $Q$ is a smooth point, so that $\OO_{X,Q}$ is a discrete valuation ring with valuation $v_Q$ and residue field of degree $\deg Q$ over $\F_q$. If $v_Q(f)\geq0$, then $\OO_{X,Q}+f\OO_{X,Q}=\OO_{X,Q}$ and the local degree is $0$. If $v_Q(f)=-n<0$, then $\OO_{X,Q}+f\OO_{X,Q}=f\OO_{X,Q}=t^{-n}\OO_{X,Q}$ for a local parameter $t$, and
\[
\deg_Q\bigl(f\OO_{X,Q}\bigr)=\dim_{\F_q}\bigl(t^{-n}\OO_{X,Q}\big/\OO_{X,Q}\bigr)=n\deg Q .
\]

Summing these local degrees over the smooth points of $X$, we obtain the contribution of all poles of $f$ lying over the smooth locus to $\deg_{\tX}(\divi_\infty(f))$, each pole counted with its multiplicity and degree.

Now let $Q$ be a non-split singularity with branch $P$. Write $\mm_P:=\mm_{\tX,P}$, $\OO_P:=\OO_{\tX,P}$, and  $\OO_Q:=\OO_{X,Q}=\F_q+\mm_P$, so that $\OO_Q\subset\OO_P$. Set $D_Q:=\OO_Q+f\OO_Q$. Recall that $\deg P=2$, so the residue field of $\OO_P$ is $\F_{q^2}$. Three cases occur.

If $f\in\OO_Q$, then $D_Q=\OO_Q$ and $\deg_QD_Q=0$; moreover $v_P(f)\geq0$, so $P$ does not occur in $\divi_\infty(f)$. This matches \eqref{eq:degformula} with $\varepsilon_Q(f)=0$.

If $f\in\OO_P\setminus\OO_Q$, then $f +\mm_P \in\F_{q}(P)\setminus\F_q$. Since $f\mm_P\subseteq\mm_P\subseteq\OO_Q$, we get $D_Q=\OO_Q+\F_q f=\F_q+\F_qf+\mm_P=\OO_P$, because $\{1+\mm_P,f+\mm_P\}$ is an $\F_q$-basis of $\F_{q}(P)=\OO_P/\mm_P$. Hence
$\deg_QD_Q=\dim_{\F_q}\bigl(\OO_P/\OO_Q\bigr)=\delta_Q=1 $,
while again $P$ does not occur in \review{$\divi_\infty(f)$, giving the case} $\varepsilon_Q(f)=1$.

If $v_P(f)=-n<0$, then $f^{-1}\in\mm_P\subseteq\OO_Q$, so $\OO_Q=f\cdot f^{-1}\OO_Q\subseteq f\OO_Q$ and $D_Q=f\OO_Q$. We claim that \review{$\OO_Q\subseteq\OO_P\subseteq f\OO_Q\subseteq f\OO_P$. Indeed,} \[
f\OO_Q=\F_qf+f\mm_P=\F_qf+\mm_P^{\,1-n}\supseteq\mm_P^{\,1-n}\supseteq\OO_P,
\] since $1-n\leq0$. As \[
\dim_{\F_q}(\OO_P/\OO_Q)=\dim_{\F_q}(f\OO_P/f\OO_Q)=1
\] and $\dim_{\F_q}(f\OO_P/\OO_P)=n\deg P=2n$, we conclude
\[
\begin{aligned}
\deg_QD_Q&=\dim_{\F_q}\bigl(f\OO_Q/\OO_Q\bigr)
=\dim_{\F_q}\bigl(f\OO_Q/\OO_P\bigr)+1\\
&=(2n-1)+1
=n\deg P ,
\end{aligned}
\]
which is exactly the contribution of $P$, with multiplicity $n$, to \review{$\deg_{\tX}(\divi_\infty(f))$, again with} $\varepsilon_Q(f)=0$.

Finally, let $Q$ be rational unibranch with rational
branch $P$, and write
$O=\OO_{X,Q}$, $B=\OO_{\tX,P}$ and $D=O+fO$.
Then $\dim_{\F_q}(B/O)=\delta_Q$.

If $v_P(f)\geq0$, then $O\subseteq D\subseteq B$ and
$0\leq\varepsilon_Q(f)=\dim_{\F_q}(D/O)\leq\delta_Q$.

If $v_P(f)=-n<0$, let $a=f^{-1}\in B$. Multiplication by $a$ induces an isomorphism
$D/O\ \simeq\ (aO+O)/aO$.
Since $B$ is a discrete valuation ring with residue field $\F_q$ and $v_P(a)=n$, the inclusions $aO\subseteq aB\subseteq B$ give
\[
\dim_{\F_q}(B/aO)
=\dim_{\F_q}(B/aB)+\dim_{\F_q}(aB/aO)
=n+\delta_Q.
\]
Here multiplication by $a$ identifies $B/O$ with $aB/aO$. On the other hand, $O\subseteq aO+O\subseteq B$, so
$0\leq\dim_{\F_q}\bigl(B/(aO+O)\bigr)\leq\delta_Q$.
Consequently,
\[
\begin{aligned}
\deg_QD&=\dim_{\F_q}(D/O)
=n+\delta_Q-\dim_{\F_q}\bigl(B/(aO+O)\bigr),\\
\varepsilon_Q(f)&=\delta_Q-\dim_{\F_q}\bigl(B/(aO+O)\bigr)\in[0,\delta_Q].
\end{aligned}
\]
The contribution of the pole divisor at $P$ is $n$, since $\deg P=1$. Summing these local computations proves \eqref{eq:degformula} and the final inequality.
\end{proof}

\subsection{Regularity at the nodes}\label{subsec:regularity}

We now return to the curve $C\subset\PP^3$ in Theorem~\ref{thm:construction}, with function field $K=\F_q(x,y)$, non-split singularities $Q$ (and $Q^-$, when $\beta\in\F_q$) and branches $P$ (and $P^-$).

\begin{lemma}\label{lem:regularity}
A function $f\in K$ belongs to $\OO_{C,Q}$ if and only if $f\in\OO_{\tC,P}$ and $f\bmod\mm_P\in\F_q$. The same holds at $Q^-$, with $P^-$ in place of $P$.
\end{lemma}

\begin{proof}
\review{By Lemma~\ref{lem:localchar}, $\OO_{C,Q}=\F_q+\mm_P$. For $f\in\OO_{\tC,P}$, its residue lies in $\F_q$ if and only if $f-c\in\mm_P$ for some $c\in\F_q$, which is equivalent to $f\in\OO_{C,Q}$. The same argument applies at $Q^-$.}
\end{proof}

\begin{corollary}\label{cor:standardfunctions}
For the coordinate functions of the model \eqref{eq:model}:
\begin{enumerate}
\item[(i)] $x\notin\OO_{C,Q}$, and $x\notin\OO_{C,Q^-}$ when $\beta\in\F_q$, since $x\bmod\mm_P=\alpha\notin\F_q$, if $\mathbb{F}_q(P)$ is viewed inside $\mathbb{F}_{q^2}(P_1)$, and likewise at $P^-$;
\item[(ii)] $y\in\OO_{C,Q}$ if and only if $\beta\in\F_q$, since $y\bmod\mm_P=\beta$, if $\mathbb{F}_q(P)$ is viewed inside $\mathbb{F}_{q^2}(P_1)$; in that case also $y\in\OO_{C,Q^-}$, since $y\bmod\mm_{P^-}=-\beta\in\F_q$;
\item[(iii)] $u=\varphi(x)$ belongs to $\OO_{C,Q}$, and to $\OO_{C,Q^-}$ \review{when $\beta\in\F_q$, since} $u\bmod\mm_P=\varphi(\alpha)=0$, so that $u\in\mm_P\subset\OO_{C,Q}$, and likewise at $P^-$.
\end{enumerate}
Moreover, by \eqref{eq:valinfty} the functions $x$, $y$ and $u$ are regular on the affine chart of $\tC$ and
$\divi_\infty(x)=2P_\infty$, 
$\divi_\infty(y)=(2g+1)P_\infty$ and  
$\divi_\infty(u)=4P_\infty$.
\end{corollary}

The following elementary observation links the vanishing of the local terms $\varepsilon_Q$ to linear equivalences on $\tC$; it will be used repeatedly.

\begin{lemma}\label{lem:linkage}
Let $Q'\in\{Q,Q^-\}$ be a non-split singularity of $C$, with branch $P'\in\tC$, and let $f\in K\setminus\F_q$ satisfying $\varepsilon_{Q'}(f)=0$. Then
$\divi_\infty(f)\sim P'+E$
for some effective divisor $E$ on $\tC$ of degree $\deg\divi_\infty(f)-2$.
\end{lemma}

\begin{proof}
By definition of $\varepsilon_{Q'}$, either $v_{P'}(f)<0$ or $f\in\OO_{C,Q'}$. In the first case $P'\leq\divi_\infty(f)$, and we may take $E:=\divi_\infty(f)-P'\geq0$; here the equivalence is an equality. In the second case, $c:=f\bmod\mm_{P'}$ lies in $\F_q$ by Lemma~\ref{lem:regularity}, and $f-c$ vanishes at $P'$; since
$\deg\divi_0(f-c)=\deg\divi_\infty(f-c)=\deg\divi_\infty(f)$,
we may write $\divi_0(f-c)=P'+E$ with $E\geq0$ of degree $\deg\divi_\infty(f)-2$, and then
$\divi_\infty(f)=\divi_\infty(f-c)\sim\divi_0(f-c)=P'+E$.
\end{proof}

\subsection{The gonality theorem}\label{subsec:gonalitythm}

Besides the nodes, the curves with $g=1$ carry the ordinary cusp $Q_\infty$ in Theorem~\ref{thm:construction}(e), whose local terms we now study.

\begin{lemma}\label{lem:cusp}
Let $g=1$ and let $Q_\infty$ be the cusp of the curve $C$ in Theorem~\ref{thm:construction}, so that $S(Q_\infty)=\langle2,3\rangle$, $\delta_{Q_\infty}=1$ and $\OO:=\OO_{C,Q_\infty}=\F_q+\mm_{\tC,P_\infty}^{2}$. Then, for every $f\in K^{\times}$, the local term of Proposition~\ref{prop:degformula2} satisfies $\varepsilon_{Q_\infty}(f)\in\{0,1\}$, and $\varepsilon_{Q_\infty}(f)=1$ whenever $v_{P_\infty}(f)=-1$. Moreover:
\begin{enumerate}
\item[(i)] if $f$ is regular at $P_\infty$, then $\varepsilon_{Q_\infty}(f)=0$ if and only if $v_{P_\infty}\bigl(f-f(P_\infty)\bigr)\neq 1$;
\item[(ii)] $\varepsilon_{Q_\infty}(x)=0$ and $\varepsilon_{Q_\infty}(y)=0$.
\end{enumerate}
\end{lemma}

\begin{proof}
Write $v:=v_{P_\infty}$, $S:=\langle2,3\rangle$, $D:=\OO+f\OO$ and $n:=\max(0,-v(f))$, so that $\varepsilon_{Q_\infty}(f)=\deg_{Q_\infty}D-n=\#\bigl(S(D)\setminus S\bigr)-n$ by Lemma~\ref{lem:gaps}, and $0\leq\varepsilon_{Q_\infty}(f)\leq\delta_{Q_\infty}=1$ by Proposition~\ref{prop:degformula2}.

If $n=0$, then $D\subseteq\tO$ and $\varepsilon_{Q_\infty}(f)=\dim_{\F_q}D/\OO\geq0$; it vanishes if and only if $f\in\OO=\F_q+\mm_{\tC,P_\infty}^{2}$, that is --- the residue field being $\F_q$ --- if and only if $v\bigl(f-f(P_\infty)\bigr)\neq1$. This proves (i).

For $n=1$, we have $S(D)\supseteq-1+S$, from the proof of Lemma~\ref{lem:gaps}, and hence $\{-1,1\}\subseteq S(D)\setminus S$. Thus $\deg_{Q_\infty}D\geq2$ and $\varepsilon_{Q_\infty}(f)\geq1$; the upper bound gives equality. The nonnegativity for all other values of $n$ already follows from Proposition~\ref{prop:degformula2}.

\review{For (ii), every element} of $D$ is of the form $a+fb$ with $a,b\in\OO$. For $f=x$ we have $v(x b)\in-2+S=\{-2,0,1,2,\dots\}$ and $v(a)\in S\subseteq\mathbb{N}$, so a negative valuation can only arise from the term $xb$ with $v(xb)=-2$, in which case $v(a+xb)=-2$; consequently $S(D)\subseteq\mathbb{Z}_{\geq-2}\setminus\{-1\}$ and $\#(S(D)\setminus S)\leq\#\{-2,1\}=2=n$, whence $\varepsilon_{Q_\infty}(x)\leq0$, so $=0$. For $f=y$ the same argument with $v(yb)\in-3+S=\{-3,-1,0,1,\dots\}$ gives $S(D)\subseteq\mathbb{Z}_{\geq-3}\setminus\{-2\}$ and $\#(S(D)\setminus S)\leq\#\{-3,-1,1\}=3=n$, whence $\varepsilon_{Q_\infty}(y)=0$.
\end{proof}

The next theorem computes gonality of curves constructed in Subsection~\ref{subsec:model};   the two-node genus-two
case is completed by Theorem~\ref{thm:char}. 


\begin{theorem}\label{thm:gonality}
The curves constructed in
Section~\ref{sec:construction} have the gonalities
listed in Table~\ref{tab:gonality}.
More precisely, the following statements hold.
\begin{enumerate}
\item[(a)] the curve $C_{\mathrm e}$ in Theorem~\ref{thm:elliptic} has $\gon(C_{\mathrm e})=2$, for either position of $\beta$;
\item[(b)] the curve $C$ in Theorem~\ref{thm:construction} with $\beta\notin\F_q$ and $2\leq g\leq3$ has $\gon(C)=3$; the same holds for the cone curves $C'$ in Theorem~\ref{thm:cone}, for every $g\geq4$ and either position of $\beta$;
\item[(c)] the curve $C$ in Theorem~\ref{thm:construction} with $g=1$ has $\gon(C)=3$, \review{for either position of $\beta$; the cusp} at infinity raises the gonality from $2$ to $3$ when $\beta\notin\F_q$;
\item[(d)] the curve $C$ in Theorem~\ref{thm:construction} with $\beta\in\F_q$ and $g=3$ has $\gon(C)=4$.
\end{enumerate}
For $\beta\in\F_q$ and $g=2$ one has $\gon(C)\in\{3,4\}$, and both values occur; Theorem~\ref{thm:char} below characterizes when each of them does. For the curves in Theorem~\ref{thm:construction} with $g\geq4$ see Remark~\ref{rem:gong4}.
\end{theorem}

\begin{proof}
Throughout we use \eqref{eq:gonmin} and the degree formula \eqref{eq:degformula}. First observe that $\gon(C)\geq2$ \review{in all cases. Indeed, if} $\deg_C\OO\langle1,f\rangle=1$ for some $f\notin\F_q$, then $\deg\divi_\infty(f)=1$, whence $[K:\F_q(f)]=1$ \cite[Theorem~1.4.11]{Sti09} and $K=\F_q(f)$ would be rational, contradicting $g\geq1$. The same argument shows that $\deg\divi_\infty(f)\geq2$ for every $f\in K\setminus\F_q$.

(a) Here $\Sing C_{\mathrm e}=\{Q\}$, with branch $P$, and $\tC$ is elliptic. Let $\kappa$ be a canonical divisor in $\tC$. By the Riemann--Roch theorem on $\tC$,
$\ell(P)=\deg P+1-g+\ell(\kappa-P)=2+0=2$,
since $\deg(\kappa-P)=-2<0$. \review{Choose $f\in L(P)\setminus\F_q$. Then} $\divi_\infty(f)$ is a nonzero divisor bounded by the prime divisor $P$, so $\divi_\infty(f)=P$, of degree $2$; in particular $v_P(f)<0$, so $\varepsilon_Q(f)=0$ and
$\deg_{C_{\mathrm e}}\OO\langle1,f\rangle=\deg\divi_\infty(f)=2$.
Hence $\gon(C_{\mathrm e})\leq2$, and therefore $\gon(C_{\mathrm e})=2$.

\review{(b) Here $\Sing C=\{Q\}$. This holds} for the curve in Theorem~\ref{thm:construction} with $\beta\notin\F_q$ and $2\leq g\leq3$, and for the cone curves in Theorem~\ref{thm:cone}, whose single node lies under $P^-$; the argument below is written for $C$ and the branch $P$, and applies verbatim to $C'$, with $P^-$ in place of $P$, since the residue of $x$ at $P^-$ is again $\alpha\notin\F_q$. 

By Corollary~\ref{cor:standardfunctions}(i), $x\in\OO_{\tC,P}\setminus\OO_{C,Q}$, so $\varepsilon_Q(x)=1$ and
$\deg_C\OO\langle1,x\rangle=\deg\divi_\infty(x)+1=2+1=3$,
whence $\gon(C)\leq3$. Suppose $\gon(C)=2$, realized by $f$. Then $\deg\divi_\infty(f)=2$ and $\varepsilon_Q(f)=0$, so $[K:\F_q(f)]=2$ and, by the uniqueness of the rational subfield of index two of $K$ for $g\geq2$ (Subsection~\ref{subsec:degtwo}), $\F_q(f)=\F_q(x)$. Thus
$f=\frac{ax+b}{cx+d}$, $ad-bc\neq0$.

The functions $ax+b$ and $cx+d$ are then $\F_q$-linearly independent, so they span the same $\F_q$-subspace of $K$ as $1$ and $x$, and consequently generate the same $\OO_{Q'}$-module at every point $Q'\in C$:
$(ax+b)\OO_{Q'}+(cx+d)\OO_{Q'}=\OO_{Q'}+x\,\OO_{Q'}$.
Dividing by $cx+d$ gives \[
\OO\langle1,f\rangle=\divi(cx+d)\cdot\OO\langle1,x\rangle\sim\OO\langle1,x\rangle,
\] and equivalent divisors have equal degrees
\[
2=\deg_C\OO\langle1,f\rangle=\deg_C\OO\langle1,x\rangle=3,
\]
a contradiction. Hence $\gon(C)=3$.

(c) Here $g=1$, so $\Sing C$ contains the cusp $Q_\infty$ besides the nodes, and every local term in Proposition~\ref{prop:degformula2} \review{is nonnegative, at the nodes} by \eqref{eq:degformula}, at the cusp by Lemma~\ref{lem:cusp}. We use throughout that $\deg 3P_\infty=3>2g-2=0$, so that Riemann--Roch gives $\ell(3P_\infty)=3$ and $L(3P_\infty)=\F_q\oplus\F_qx\oplus\F_qy$, the three functions having pole orders $0$, $2$ and $3$ at $P_\infty$; in particular $\divi_\infty(y)=3P_\infty$ and $\divi_\infty(x)=2P_\infty$.

Assume first $\beta\in\F_q$, so that $\Sing C=\{Q,Q^-,Q_\infty\}$ and $\pi=4$. By Corollary~\ref{cor:standardfunctions}(ii), $y\in\OO_{C,Q}\cap\OO_{C,Q^-}$, so $\varepsilon_Q(y)=\varepsilon_{Q^-}(y)=0$, while $\varepsilon_{Q_\infty}(y)=0$ by Lemma~\ref{lem:cusp}(ii); hence
$\deg_C\OO\langle1,y\rangle=\deg\divi_\infty(y)=3$,
whence $\gon(C)\leq3$.

Suppose $\gon(C)=2$, \review{realized by $f$. By the nonnegativity} of the local terms, $\deg\divi_\infty(f)=2$ and $\varepsilon_Q(f)=\varepsilon_{Q^-}(f)=0$. Applying Lemma~\ref{lem:linkage} at each node, with $E$ of degree $0$, we obtain
$P\sim\divi_\infty(f)\sim P^-$.

We claim this is impossible. Pass to the constant field extension $\tC\otimes\F_{q^2}$ and apply the conorm $\Con=\Con_{K\F_{q^2}|K}$ \review{of \cite[Lemma~5.1.9]{Sti09}. Since} $v_{P'}(h)=e(P'|\,\cdot\,)\,v(h)$ for $h\in K$, conorms of principal divisors are principal. So $P\sim P^-$ would imply
$P_1+P_2=\Con(P)\sim\Con(P^-)=P_1^-+P_2^-$ 
on the elliptic curve $\tC\otimes\F_{q^2}$. 

Recall from Lemma~\ref{lem:lambda} that here $\lambda=\beta$ is constant, so by Lemma~\ref{lem:h}
$f(x)-\beta^2=\varphi(x)h(x)$, with $h$ having degree $2g-1=1$, that is, $h(x)=h_1(x-x_3)$, $h_1\in\F_q^{\times}$ and $\ x_3\in\F_q$. Since $f-\beta^2=\varphi h$ is separable ($\varphi$ is irreducible, and its roots $\alpha,\alpha^q\notin\F_q$ differ from the rational root $x_3$ of $h$) and $\beta\neq0$, the functions $y-\beta$ and $y+\beta$ have simple zeros and
$\divi(y-\beta)=P_1+P_2+P_3-3P_\infty$,
$\divi(y+\beta)=P_1^-+P_2^-+P_3^--3P_\infty$,
where $P_3:=(x_3,\beta)$ and $P_3^-:=(x_3,-\beta)$; moreover $\divi(x-x_3)=P_3+P_3^--2P_\infty$. 

Hence 
$P_1+P_2\sim3P_\infty-P_3\sim P_\infty+P_3^-$,
$P_1^-+P_2^-\sim P_\infty+P_3$,
and $P_1+P_2\sim P_1^-+P_2^-$ would give $P_3^-\sim P_3$. On an elliptic curve, each divisor class of degree one contains a unique rational point \cite[Proposition~6.1.6]{Sti09}, so $P_3^-=P_3$, that is, \review{$\beta=-\beta$, which is impossible since} $q$ is odd and $\beta\neq0$. Hence $\gon(C)=3$ when $\beta\in\F_q$.

Assume now $\beta\notin\F_q$, so that $\Sing C=\{Q,Q_\infty\}$ and $\pi=3$. By Corollary~\ref{cor:standardfunctions}(i) and Lemma~\ref{lem:cusp}(ii),
\review{\[
\deg_C\OO\langle1,x\rangle=\deg\divi_\infty(x)+\varepsilon_Q(x)+\varepsilon_{Q_\infty}(x)=2+1+0=3,
\]}
whence $\gon(C)\leq3$. Suppose $\gon(C)=2$, \review{realized by $f$. By nonnegativity} again, $\deg\divi_\infty(f)=2$, $\varepsilon_Q(f)=0$ and $\varepsilon_{Q_\infty}(f)=0$. 

We claim that this forces $P\sim2P_\infty$ on $\tC$. Indeed, by Lemma~\ref{lem:cusp}, $\varepsilon_{Q_\infty}(f)=0$ excludes $v_{P_\infty}(f)=-1$; so either $v_{P_\infty}(f)=-2$, in which case $\divi_\infty(f)=2P_\infty$, or $f$ is regular at $P_\infty$ and, by Lemma~\ref{lem:cusp}(i), $v_{P_\infty}\bigl(f-f(P_\infty)\bigr)\geq2$, so that the zero divisor of $f-f(P_\infty)$, of degree $2$, equals $2P_\infty$ and again $\divi_\infty(f)=\divi_\infty\bigl(f-f(P_\infty)\bigr)\sim2P_\infty$. \review{On the other hand, consider the node. If} $v_P(f)<0$, then $P\leq\divi_\infty(f)$, so $\divi_\infty(f)=P$ by degrees; if $f$ is regular at $P$, then Lemma~\ref{lem:linkage} with $E$ of degree $0$ gives $P\sim\divi_\infty(f)$. In all cases $P\sim\divi_\infty(f)\sim2P_\infty$, proving the claim. Passing to $\tC\otimes\F_{q^2}$ by the conorm, as in the previous case, $P\sim2P_\infty$ yields $P_1+P_2\sim2P_\infty$. But $\divi(x-\alpha)=P_1+P_1^--2P_\infty$ on $\tC\otimes\F_{q^2}$, so $2P_\infty-P_1\sim P_1^-$, and the uniqueness of the rational point in a divisor class of degree one on an elliptic curve \cite[Proposition~6.1.6]{Sti09} forces $P_2=P_1^-=(\alpha,-\beta)$, that is, \review{$\alpha^q=\alpha$, which is impossible since} $\alpha\notin\F_q$. Hence $\gon(C)=3$ also when $\beta\notin\F_q$.

(d) Assume $\beta\in\F_q$ and $2\leq g\leq3$, so that $\Sing C=\{Q,Q^-\}$. By Corollary~\ref{cor:standardfunctions}(iii), $u=\varphi(x)\in\OO_{C,Q}\cap\OO_{C,Q^-}$, so $\varepsilon_Q(u)=\varepsilon_{Q^-}(u)=0$ and
$\deg_C\OO\langle1,u\rangle=\deg\divi_\infty(u)=4$, that is, 
$\gon(C)\leq4$. 

\review{Next, we show that $\gon(C)>2$. If} $\deg_C\OO\langle1,f\rangle=2$, then $\deg\divi_\infty(f)=2$, and exactly as in case (b) the uniqueness of the rational subfield of index two gives $\F_q(f)=\F_q(x)$ and $\OO\langle1,f\rangle\sim\OO\langle1,x\rangle$, whose degree is now
\[
\deg\divi_\infty(x)+\varepsilon_Q(x)+\varepsilon_{Q^-}(x)=2+1+1=4\neq2,
\]
a contradiction.

Now let $g=3$ and suppose $\gon(C)=3$, which is realized by some $f$ satisfying
$3=\deg\divi_\infty(f)+\varepsilon_Q(f)+\varepsilon_{Q^-}(f)$,
with 
$\deg\divi_\infty(f)\geq2$.

The value $\deg\divi_\infty(f)=2$ \review{is impossible, since it would force}, as above, $\OO\langle1,f\rangle\sim\OO\langle1,x\rangle$ of degree $4\neq3$. Hence $[K:\F_q(f)]=\deg\divi_\infty(f)=3$. We may apply Castelnuovo's inequality \cite[Theorem~3.11.3]{Sti09} to the subfields $F_1=\F_q(x)$ and $F_2=\F_q(f)$ of $K$. Note that their compositum is $K$, since $[K:F_1]=2$ and $[K:F_2]=3$. 
With $n_1=2$, $g_1=0$, $n_2=3$, $g_2=0$, Castelnuovo's inequality gives
$g\leq n_1g_2+n_2g_1+(n_1-1)(n_2-1)=2$,
contradicting $g=3$. Hence $\gon(C)=4$ for $g=3$.

Finally, for $g=2$ the two bounds just proved give $\gon(C)\in\{3,4\}$; both values occur, by Examples~\ref{ex:F3trigonal} and~\ref{ex:F3tetragonal} below.
\end{proof}

\begin{remark}\label{rem:gong4}
For the curves in Theorem~\ref{thm:construction} with $g\geq4$, Proposition~\ref{prop:degformula2} gives
$0\leq\varepsilon_{Q_\infty}(f)\leq\delta_{Q_\infty}.$
Thus the singularity at infinity cannot lower the degree relative to the normalization. \review{In particular, $\gon(C)\geq3$. Indeed, if} a pencil $\OO\langle1,f\rangle$ had degree two, the nonnegativity of all local contributions would force $\deg\divi_\infty(f)=2$. The uniqueness of the hyperelliptic pencil would then give $\OO\langle1,f\rangle\sim\OO\langle1,x\rangle$, as in the proof of Theorem~\ref{thm:gonality}(b). But $\varepsilon_Q(x)=1$ and $\varepsilon_{Q_\infty}(x)\geq0$, so $\deg_C\OO\langle1,x\rangle\geq3$, a contradiction. We do not determine the exact gonality of these projected models here; this requires further control of the local degree contributions at infinity. By contrast, the nodal cone models of Theorem~\ref{thm:cone}, constructed from the same data, are trigonal over $\F_q$ for every $g\geq4$.
\end{remark}

\subsection{The case \texorpdfstring{$g=2$, $\beta\in\F_q$}{g=2, beta rational}}\label{subsec:g2case}

In particular, $\pi=4$, $\Sing C=\{Q,Q^-\}$, $\deg\kappa=2$ and $\ell(\kappa)=2$, where as before $\kappa$ denotes a canonical divisor on $\tC$.

\begin{lemma}\label{lem:bpf}
Let $D\geq0$ be a divisor of degree $3$ on $\tC$. Then $\ell(D)=2$, and the complete linear system $|D|$ is base-point-free if and only if $D=\divi_\infty(h)$ for some $h\in K$.
\end{lemma}

\begin{proof}
By Riemann--Roch, $\ell(D)=3+1-2+\ell(\kappa-D)=2$, since $\deg(\kappa-D)=-1<0$.

Suppose $D=\divi_\infty(h)$, and let $R$ be a base point of $|D|$, that is, a closed point with $L(D-R)=L(D)$. Then $h\in L(D-R)$, so $\divi_\infty(h)\leq D-R$, and taking degrees yields the contradiction $3\leq3-\deg R$.

Conversely, suppose $|D|$ is base-point-free and pick $h\in L(D)\setminus\F_q$, so that $L(D)=\F_q\oplus\F_qh$ and $\divi_\infty(h)\leq D$. If $\divi_\infty(h)\neq D$, choose a closed point \review{$R\leq D-\divi_\infty(h)$. Then} $1,h\in L(D-R)$, so $\ell(D-R)=2=\ell(D)$ and $R$ is a base point, a contradiction. Hence $\divi_\infty(h)=D$.
\end{proof}

For $g=2$, both nodes must contribute zero to a pencil of degree three. The following criterion expresses this requirement using two fibres of a complete degree-three system on $\tC$. Its second condition excludes the systems with a fixed rational point. All linear equivalences and rational points in the criterion are taken over $\F_q$, which is essential for the distinction between the two examples below.

\begin{theorem}\label{thm:char}
Assume $\beta\in\F_q$ and $g=2$. Then $\gon(C)=3$ if and only if there exists a divisor $D\geq0$ of degree $3$ on $\tC$ such that:
\begin{enumerate}
\item[(1)] $D\sim P+\widetilde R$ and $D\sim P^-+\widehat R$ for some rational points $\widetilde R,\widehat R\in\tC(\F_q)$;
\item[(2)] $D\not\sim\kappa+R$ for every rational point $R\in\tC(\F_q)$.
\end{enumerate}
Otherwise $\gon(C)=4$.
\end{theorem}

\begin{proof}
By Theorem~\ref{thm:gonality}, $\gon(C)\in\{3,4\}$, and $\gon(C)=3$ holds precisely when $\deg_C\OO\langle1,f\rangle=3$ for some $f\in K\setminus\F_q$.

($\Leftarrow$) Let $D$ be as in the statement. We first show that $|D|$ is base-point-free. Since $D\geq0$ and $1\in L(D)$, any base point $R$ must satisfy $R\leq D$. 

If $R$ has $\deg R\geq2$, then $0\leq\deg(D-R)\leq1\leq2g-2$, and Clifford's theorem \cite[Theorem~1.6.13]{Sti09} gives
$\ell(D-R)\leq1+\tfrac{1}{2}\deg(D-R)\leq\tfrac{3}{2}$,
so $\ell(D-R)\leq1<2=\ell(D)$ and $R$ is not a base point. 

If $R\in\tC(\F_q)$ is a rational base point, then $\ell(D-R)=\ell(D)=2$ with $\deg(D-R)=2$, and Riemann--Roch gives $\ell(\kappa-D+R)=1$ with $\deg(\kappa-D+R)=0$; a divisor of degree zero with a nonzero section is principal, so $D\sim\kappa+R$, contradicting (2). 

Hence $|D|$ is base-point-free and, by Lemma~\ref{lem:bpf}, $D=\divi_\infty(f)$ for some $f$, with $L(D)=\F_q\oplus\F_qf$.

From $D\sim P+\widetilde R$ we get $P+\widetilde R=D+\divi(z)$ for some $z\in L(D)\setminus\{0\}$; write $z=a+bf$ with $a,b\in\F_q$. If $b=0$, then $\divi(z)=0$ and $\divi_\infty(f)=D=P+\widetilde R\geq P$, so $v_P(f)<0$ and $\varepsilon_Q(f)=0$. If $b\neq0$, then $\divi_\infty(z)=\divi_\infty(f)=D$, whence $\divi_0(z)=P+\widetilde R$ and $z=a+bf$ vanishes at $P$. Thus $f\equiv-a/b\in\F_q$ modulo $\mm_P$ and $f\in\OO_{C,Q}$ by Lemma~\ref{lem:regularity}, so again $\varepsilon_Q(f)=0$. The condition $D\sim P^-+\widehat R$ gives $\varepsilon_{Q^-}(f)=0$ in the same way. Therefore
$\deg_C\OO\langle1,f\rangle=\deg\divi_\infty(f)=3$ and $\gon(C)=3$.

($\Rightarrow$) Assume $\gon(C)=3$, realized by $f$. As in the proof of Theorem~\ref{thm:gonality}(d),
$\deg\divi_\infty(f)=3,$
$\varepsilon_Q(f)=\varepsilon_{Q^-}(f)=0$. Set $D:=\divi_\infty(f)\geq0$. 

By Lemma~\ref{lem:linkage} applied at $Q$ and at $Q^-$, $D\sim P+\widetilde R$ and $D\sim P^-+\widehat R$ with $\widetilde R,\widehat R$ effective of degree $1$, that is, condition (1) holds. For (2), if $D\sim\kappa+R$ for some $R\in\tC(\F_q)$, then $\ell(D-R)=\ell(\kappa)=2=\ell(D)$, so $R$ would be a base point of $|D|$; but $|D|$ is base-point-free by Lemma~\ref{lem:bpf}, since $D=\divi_\infty(f)$.
\end{proof}

We now exhibit both behaviours over $\F_3$.

\begin{example}[$q=3$, $\gon(C)=3$]\label{ex:F3trigonal}
Let $\tC\colon y^2=F(x)$ be defined over $\F_3$ by the squarefree polynomial $F(x)=x^5-x+1$. 
Thus $\tC$ is smooth of genus $2$. Moreover $\tC$ is geometrically irreducible, since $\deg F=5$ is odd. As $F(0)=F(1)=F(2)=1$ is a square, the rational points of $\tC$ are the six affine points $(i,\pm1)$, $i\in\{0,1,2\}$, together with $P_\infty$.

Take a root $\alpha\in\F_9$ of the irreducible polynomial $\varphi(x)=x^2+1\in\F_3[x]$. Then $\alpha^4=1$,
$F(\alpha)=\alpha^5-\alpha+1=\alpha-\alpha+1=1$,
and $\beta=1\in\F_3$. Theorem~\ref{thm:construction}(d) produces a curve $C$ with two non-split singularities and $\pi=4$, and so $\gon(C)\in\{3,4\}$ by Theorem~\ref{thm:gonality}. 

Consider
$f:=\frac{y+1}{x^2-1}\in K$. 
At infinity one has $v_{P_\infty}(f)=-5+4=-1$. At the affine rational points we have
$(y-1)(y+1)=y^2-1=F(x)-1=x^5-x=x(x-1)(x+1)(x^2+1)$. 
Since none of the six points is a Weierstrass point, at each point $(i,\pm1)$ exactly one of $y\mp1$ \review{has a simple zero, with $v(y-1)=1$} at $(i,1)$ and $v(y+1)=1$ at $(i,-1)$, while $v(x^2-1)=1$ at the four points with $i\in\{1,2\}$. Hence $f$ is regular at $(1,-1)$ and $(2,-1)$, it has simple poles at $(1,1)$ and $(2,1)$, and it is regular at every other affine point of $\tC$ (at the points above $x^2+1$, namely $P$ and $P^-$, the denominator satisfies $\alpha^2-1=-2\neq0$). Therefore
$\divi_\infty(f)=P_\infty+(1,1)+(2,1)$, $\deg\divi_\infty(f)=3$.

At the nodes, the values of $f$ along the branches are
$f(P_1)=\frac{1+1}{\alpha^2-1}=\frac{2}{-2}=-1\in\F_3$, 
$f(P_1^-)=\frac{-1+1}{\alpha^2-1}=0\in\F_3$.
So $f\in\OO_{C,Q}\cap\OO_{C,Q^-}$ by Lemma~\ref{lem:regularity} and $\varepsilon_Q(f)=\varepsilon_{Q^-}(f)=0$. Now the degree formula \eqref{eq:degformula} yields $\deg_C\OO\langle1,f\rangle=3$, whence
$
\gon(C)=3,
$
i.e., the curve $C$ is trigonal.
\end{example}

\begin{example}[$q=3$, $\gon(C)=4$]\label{ex:F3tetragonal}
Let $\tC\colon y^2=F(x)$ be given over $\F_3$ by the squarefree polynomial $F(x)=x^5-x^3-1$. 
Thus $\tC$ is smooth of genus $2$ and geometrically irreducible, with a single rational point at infinity. There are no affine rational points, since $F(0)=F(1)=F(2)=2$ is not a square in $\F_3$. Thus
$\tC(\F_3)=\{P_\infty\}$.

Take a root $\alpha$ of the irreducible polynomial $x^2-x-1\in\F_3[x]$, so that $\alpha^2=\alpha+1$, $\alpha^3=2\alpha+1$, $\alpha^4=2$, $\alpha^5=2\alpha$, 
$F(\alpha)=2\alpha-(2\alpha+1)-1=-2=1$ and $\beta=1\in\F_3$. Now Theorem~\ref{thm:construction}(d) applies, and so $C$ has two non-split singularities, $\pi=4$, and $\gon(C)\in\{3,4\}$.

We show that no divisor $D$ as in Theorem~\ref{thm:char} exists, so that $\gon(C)=4$. Since $\tC(\F_3)=\{P_\infty\}$, condition (1) in the theorem implies $\widetilde R=\widehat R=P_\infty$, hence
$D\sim P+P_\infty\sim P^-+P_\infty$
and therefore
$P\sim P^-$.

We claim the latter is impossible. Note that $P\neq P^-$, since $\beta\neq0$. If $P\sim P^-$, pick $z$ with $\divi(z)=P^--P$. Then $z\in L(P)\setminus\F_3$ and $\ell(P)\geq2$. Riemann--Roch gives $\ell(P)=2+1-2+\ell(\kappa-P)$, so $\ell(\kappa-P)\geq1$ with $\deg(\kappa-P)=0$, forcing $\kappa-P\sim0$, i.e. 
$P\sim\kappa$.
On the other hand, $\deg 2P_\infty=2=2g-2$ and $\ell(2P_\infty)=2=g$, as $L(2P_\infty)=\F_3\oplus\F_3x$, which by Riemann--Roch means that the divisor $2P_\infty$ is canonical. Thus $P\sim2P_\infty$, and there is a function $z=a+bx\in L(2P_\infty)$, $b\neq0$, with $\divi(z)=P-2P_\infty$. Writing $z=b(x-c)$ with $c\in\F_3$, we get $\divi_0(x-c)=P$, and so $x\equiv c\in\F_3$ modulo $\mm_P$, contradicting $x\bmod\mm_P=\alpha\notin\F_3$. This shows that $\gon(C)=4$, i.e., the curve $C$ is tetragonal.
\end{example}

\begin{remark}\label{rem:tetragonalcriterion}
The argument of Example~\ref{ex:F3tetragonal}
applies to any genus-two curve $\tC$ with a model
\eqref{eq:model} and
$\tC(\F_q)=\{P_\infty\}$.
Thus, whenever a closed point $P$ satisfies the
standing hypotheses with $\beta\in\F_q$,
the curve obtained from
Theorem~\ref{thm:construction} is tetragonal.
\end{remark}

\section*{Declaration of generative AI and AI-assisted
technologies}

The authors used ChatGPT (OpenAI), Claude (Anthropic),
and Gemini (Google) during the research and preparation
of this manuscript. These tools assisted with language
editing and manuscript organization, as well as with
the development of mathematical arguments and examples.
Specifically, AI assistance contributed to finding the
affine function underlying the construction in
Theorem~3.3, adapting that construction to the elliptic
case in Theorem~3.4, devising the proof of Lemma~2.8,
and developing all examples in Subsection~3.5 and
Examples~4.12 and~4.13.

The authors independently verified all mathematical
arguments, computations, and references incorporated
into the manuscript, and reviewed and edited the final
text. They take full responsibility for the results
and the entire content of the article.

\end{document}